\documentclass[final,notitlepage,12pt,reqno,tbtags]{amsart}
\usepackage{graphicx}
\usepackage{url}
\usepackage{mathrsfs}

\usepackage{rotating}

\usepackage{xcolor}
\usepackage[pstarrows]{pict2e}

\usepackage{graphicx}
\usepackage{tikz}

\usepackage{pgfplots}
\makeatletter
\let\I\@undefined
\makeatother
\usepackage{geometry}
\usepackage{indentfirst}
\usepackage[normalem]{ulem}
\usepackage{float}
\usepackage{amsthm}

\usepackage{enumitem}[2011/09/28]
\setenumerate{align=left, leftmargin=0pt,labelsep=.5em, labelindent=0\parindent,listparindent=\parindent,itemindent=*}
\usepackage{array}
\usepackage{empheq}
\usepackage{natbib}
\usepackage[all]{xy}

\usepackage{caption}[2012/02/19]

\usepackage{mathptmx}
\DeclareSymbolFont{operators}{OT1}{txr}{m}{n}
\SetSymbolFont{operators}{bold}{OT1}{txr}{bx}{n}
\def\operator@font{\mathgroup\symoperators}
\DeclareSymbolFont{italic}{OT1}{txr}{m}{it}
\SetSymbolFont{italic}{bold}{OT1}{txr}{bx}{it}
\DeclareSymbolFontAlphabet{\mathrm}{operators}
\DeclareMathAlphabet{\mathbf}{OT1}{txr}{bx}{n}
\DeclareMathAlphabet{\mathit}{OT1}{txr}{m}{it}
\SetMathAlphabet{\mathit}{bold}{OT1}{txr}{bx}{it}
\DeclareSymbolFont{letters}{OML}{txmi}{m}{it}
\SetSymbolFont{letters}{bold}{OML}{txmi}{bx}{it}
\DeclareFontSubstitution{OML}{txmi}{m}{it}
\DeclareSymbolFont{lettersA}{U}{txmia}{m}{it}
\SetSymbolFont{lettersA}{bold}{U}{txmia}{bx}{it}
\DeclareFontSubstitution{U}{txmia}{m}{it}
\DeclareSymbolFontAlphabet{\mathfrak}{lettersA}
\DeclareSymbolFont{symbols}{OMS}{txsy}{m}{n}
\SetSymbolFont{symbols}{bold}{OMS}{txsy}{bx}{n}
\DeclareFontSubstitution{OMS}{txsy}{m}{n}

\usepackage{amssymb,bm,amsmath}
\usepackage{amsfonts}

\usepackage[OT2,T1,T2A]{fontenc}

\usepackage[english,french,german,russian]{babel}
\usepackage{appendix}

\DeclareMathOperator{\Span}{span}

\DeclareMathOperator{\D}{d}
\DeclareMathOperator{\I}{Im}
\DeclareMathOperator{\R}{Re}

\bibpunct{[}{]}{,}{n}{}{;}

\theoremstyle{plain}
\newtheorem{theorem}{Theorem}[section]

\newtheorem{lemma}[theorem]{Lemma}

\newbox\shell
\newcommand{\dia}[2]{\setbox\shell=\hbox{\begin{picture}(180,120)(-90,-60)#1
\put(-90,-60){\makebox(180,120)[b]{\large #2}}\end{picture}}\dimen0=\ht
\shell\multiply\dimen0by7\divide\dimen0by16\raise-\dimen0\box\shell\hfill}

\newcommand{\vtx}{\circle*{10}}

\theoremstyle{definition}

\numberwithin{equation}{section}

\usepackage{color}

\begin{document}

\pagenumbering{roman}
\selectlanguage{english}
\title{On Borwein's  conjectures for planar uniform random walks}
\author{Yajun Zhou}
\address{Program in Applied and Computational Mathematics (PACM), Princeton University, Princeton, NJ 08544; Academy of Advanced Interdisciplinary Studies (AAIS), Peking University, Beijing 100871, P. R. China} \email{yajunz@math.princeton.edu, yajun.zhou.1982@pku.edu.cn}
\thanks{\indent * This research was supported in  part  by the Applied Mathematics Program within the Department of Energy
(DOE) Office of Advanced Scientific Computing Research (ASCR) as part of the Collaboratory on
Mathematics for Mesoscopic Modeling of Materials (CM4)}
\date{\today}

\maketitle

\begin{center}\footnotesize{\textit{In memoriam Jonathan M. Borwein }(1951--2016)}\end{center}\vspace{-1em}

\begin{abstract}
   Let $ p_n(x)=\int_0^\infty J_0(xt)[J_0(t)]^n xt\D t$ be Kluyver's probability density for $n$-step uniform random walks in the Euclidean plane. Through connection to a similar problem in 2-dimensional quantum field theory, we evaluate the third-order derivative $ p_5'''(0^{+})$ in closed form, thereby giving a new proof for a conjecture of J. M. Borwein. By further analogies to Feynman diagrams in quantum field theory, we demonstrate that  $ p_n(x),0\leq x\leq 1$  admits a uniformly convergent Maclaurin expansion  for all odd integers $ n\geq5$, thus settling another conjecture of Borwein.  \end{abstract}


\pagenumbering{arabic}

\section{Introduction}
Following Pearson \cite{Pearson1905a,Pearson1905b} and Rayleigh \cite{Rayleigh1905}, we consider a rambler walking in the Euclidean plane, taking $n$ consecutive steps of unit lengths, aiming at uniformly distributed random directions (Fig.~\ref{fig:p5}\textit{a}). For $n\in\mathbb Z_{>1}$, the distance $x$ traveled by such a random walker is characterized by Kluyver's probability density function \cite{Kluyver1906}:\begin{align}
p_{n}(x)=\int_0^\infty J_0(xt)[J_0(t)]^n xt\D t,\label{eq:Kluyver_pn}
\end{align} where  $ J_0(t):=\frac{2}{\pi}\int_0^{\pi/2}\cos(t\cos\varphi)\D\varphi$ is the Bessel function of the first kind and\ zeroth order. According to the statistical interpretation, the function  $p_n(x)$ is supported on $[0,n]$. The result \begin{align}
p_2(x)=\frac{2}{\pi\sqrt{4-x^{2}}},\quad 0\leq x<2
\end{align}is classical. The analytic and arithmetic properties of the 3-step density   $ p_3(x)$ and the 4-step density $p_4(x)$ for \textit{planar uniform random walks} have been thoroughly explored by Borwein and coworkers \cite{BNSW2011,BSWZ2012,BSW2013}.

In \cite[\S5]{BSWZ2012}, Borwein \textit{et al.}~investigated the Maclaurin expansion $ p_5(x)=\sum^\infty_{k=0}r_{5,k}x^{2k+1}$ for small and positive $x$, arriving at a closed-form evaluation \cite[Theorem 5.1]{BSWZ2012} of the leading Taylor coefficient via special values of  Euler's gamma function:\footnote{In this work, we write $ f'(0^+) $ for the  one-sided limit of the derivative, namely,  $ \lim_{x\to0^+}f'(x)$. Since the functions we study are  right-continuous at the origin, \textit{i.e.}~$ f(0)=f(0^+):=\lim_{x\to0^+}f(x)$, the value of $f'(0^+) $ also agrees with the derivative from the right $ f'_+(0):=\lim_{x\to0^+}[f(x)-f(0)]/x$, according to Lagrange's mean value theorem.}\begin{align}
r_{5,0}^{\phantom'}=p'_5(0^{+})=\frac{\sqrt{5}}{40\pi^4}\Gamma\left( \frac{1}{15} \right)\Gamma\left( \frac{2}{15} \right)\Gamma\left( \frac{4}{15} \right)\Gamma\left( \frac{8}{15} \right)\label{eq:r50}
\end{align} and a conjectural relation \cite[cf.][(5.3)]{BSWZ2012}\begin{align}
r_{5,1}\overset{?}=\frac{13}{225}r_{5,0}-\frac{2}{5\pi^4r_{5,0}}.\label{eq:B_conj}
\end{align}During a recent study of random walks in 4-dimensional Euclidean space, Borwein \textit{et al.}~have proposed an equivalent form  \cite[(89)]{BSV2016} of the conjecture above:\begin{align}
8\int_{0}^\infty[J_1(t)]^5\frac{\D t}{t^{2}}=\frac{1}{6}r_{5,0}+\frac{105}{16\pi^4r_{5,0}},\label{eq:B_conj'}
\end{align}  where $ J_1(t)=-\D J_0(t)/\D t$.
The integral on the left-hand side of \eqref{eq:B_conj'} can be evaluated in closed form \cite[Example 4.15]{BSV2016}, so the original conjecture in \eqref{eq:B_conj} has been verified by a connection between 2-dimensional and 4-dimensional random walks \cite[Theorem 4.17]{BSV2016}.

Let \begin{align}
W_{n}(s):=\int_{0}^1\D  t_1\cdots\int_0^1\D t_n\left\vert \sum^n_{k=1}e^{2\pi i t_k}\right\vert^s
\end{align}be Pearson's $n$-step \textit{ramble integral} for complex-valued $s$.  For $n\in\mathbb Z_{>1} $ and $\R s>0$, the convergent ramble integral is related to Kluyver's probability density by a moment formula  $ W_n(s)=\int_0^\infty x^s p_n(x)\D x$ \cite[(2.3)]{BSWZ2012}. In \cite{BNSW2011,BSW2013}, Borwein \textit{et al.}~have studied analytic continuations of $W_n(s)$,  showing that the only possible singularities are poles at certain negative integers, and the order of each pole is at most $2$. \textit{En route} to giving a partial proof of a sum rule involving the ramble integrals (see \cite[Conjecture 1]{BNSW2011} and \cite[Conjecture 1.1]{BSW2013}):\begin{align}
W_{2j+2}(s)\overset{?}=\sum_{m=0}^\infty\left[\frac{\Gamma\left(\frac{s}{2}+1\right)}{\Gamma(m+1)\Gamma\left(\frac{s}{2}-m+1\right)}\right]^2W_{2j+1}(s-2m ),\label{eq:Wns_sum_rule}
\end{align} where $ j\in\mathbb Z_{>0},s\in\mathbb C\smallsetminus\mathbb Z$, Borwein--Straub--Wan conjectured that all the poles of  $ W_n(s)$ are \textit{simple} when $n$ is an odd number greater than $1$ \cite[Conjecture 4.1]{BSW2013}. While this ``simple pole conjecture'' has been  tested on individual odd integers up to $n=45$ \cite{BSW2013} using recurrence formulae for  ramble integrals, we have not been able to locate a proof for the general cases in previous literature.

In \S\ref{sec:new_pf_B_conj}, we give a new proof of Borwein's conjecture stated in \eqref{eq:B_conj}. Instead of going to 4-dimensional space, we use a \textit{Wick rotation} to establish a link between $ p_5(x)$ and a Bessel moment $ \int_0^\infty I_0(xt)I_{0}(t)[K_0(t)]^4 t\D t$ occurring in 2-dimensional quantum field theory, where  $ I_0(t)=\frac{1}{\pi}\int_0^\pi e^{t\cos\theta}\D\theta$ and $ K_0(t)=\int_0^\infty e^{-t\cosh u}\D u$  are modified Bessel functions of zeroth order. This link allows us to compute $ p_5'''(0^{+})$  in terms of the Bessel moment $ \int_0^\infty I_{0}(t)[K_0(t)]^4 t^{3}\D t$, whose value has been conjectured by Broadhurst in the famous paper of Bailey--Borwein--Broadhurst--Glasser \cite[(96)]{Bailey2008}, and confirmed in our recent work \cite[\S3]{Zhou2017WEF}.

In \S\ref{sec:Maclaurin}, we study the Maclaurin expansion of $p_5(x) $ in full, by further exploiting  the connection between  $ p_5(x)$ and  $ \int_0^\infty I_0(xt)I_{0}(t)[K_0(t)]^4 t\D t$. This results in a proof for the \textit{strict positivity} of all the Taylor coefficients ($r_{5,k}>0 $ for $ k\in\mathbb Z_{\geq0}$), and subsequently, uniform convergence of the Maclaurin series $\sum^\infty_{k=0}r_{5,k}x^{2k+1} $ on $ [-3,3]$. In addition, we shed new light on the \textit{Pearson--Fettis phenomenon}, namely, the approximate linearity of $p_5(x),0\leq x\leq 1 $ (see Fig.~\ref{fig:p5}\textit{b}).

In \S\ref{sec:6_8}, we  set  the leading asymptotic behavior for $p_5(x) $, $p_6(x)$ and $ p_8(x)$   in a unified framework, and analyze the Maclaurin expansion of $ p_7(x)$. For the first task, we show that $ p_5'(0^{+})$, $ p_6'(0^{+})$ and $ p_8'(0^{+})$ are representable through \textit{critical values} of certain modular $ L$-series, which are also  related to several Bessel moments in  2-dimensional quantum field theory. Towards this end, we revisit some important conjectures of Broadhurst \cite{Broadhurst2013MZV,Broadhurst2016} that have been recently proven  \cite[\S\S4--5]{Zhou2017WEF}. For the second task, we relate $ p_7(x)$ to  $ \int_0^\infty I_0(xt)I_{0}(t)[K_0(t)]^6 t\D t$ and    $ \int_0^\infty I_0(xt)[I_{0}(t)]^{3}[K_0(t)]^4 t\D t$, by analogy to the    Maclaurin expansion of $ p_5(x)$.

In \S\ref{sec:pn_Mac}, we further generalize our methods  to the analysis of  $ p_{2j+1}(x),0\leq x<1$, where $j$ is a positive integer. We show that all such probability densities are representable as  $ \mathbb Q$-linear combinations of Feynman diagrams $ \int_0^\infty I_0(xt) I_0(t)^{2m+1}[K_0(t)/\pi]^{2(j-m)}xt\D t$ where $ m\in\mathbb Z\cap[0,(j-1)/2 ]$. This establishes, \textit{a fortiori}, that   $ p_{2j+1}(x), 0\leq x< 1$  admits a  convergent Maclaurin expansion,  that the corresponding ramble integral $ W_{2j+1}(s)$ has only \textit{simple poles}, and that the sum rule in \eqref{eq:Wns_sum_rule} is true. \section{A new proof of Borwein's conjecture for 5-step  random walks\label{sec:new_pf_B_conj}}For $ x\in[0,3]$, the Bessel moment $ \int_0^\infty I_0(xt)I_{0}(t)[K_0(t)]^4 t\D t$  evaluates (up to a normalizing constant) the following \textit{Feynman diagram} in 2-dimensional quantum field theory:\begin{align}
\;\;\;\;\;
\dia{\put(-50,0){\line(-1,-1){50}}
\put(-50,0){\line(-1,1){50}}
\put(50,0){\line(1,1){50}}
\put(50,0){\line(1,-1){50}}
\put(-100,45){\begin{rotate}{-45}{$^1$}\end{rotate}}
\put(-80,-75){\begin{rotate}{45}{$^{x }$}\end{rotate}}
\put(0,0){\circle{100}}
\qbezier(-50, 0)(0, 35)(50, 0)
\qbezier(-50, 0)(0, -35)(50, 0)
\put(50,0){\vtx}
\put(-50,0){\vtx}
\put(90,35){\begin{rotate}{45}{$^{x}$}\end{rotate}}
\put(65,-60){\begin{rotate}{-45}{$^{1}$}\end{rotate}}
\put(-8,45){$^1$}
\put(-8,10){$^{1 }$}
\put(-8,-45){$^{1 }$}
\put(-8,-80){$^{1 }$}
}{}
\end{align}where all the four internal lines carry unit masses (corresponding to $ [K_0(t)]^4$), one pair of external legs carry unit mass (corresponding to $ I_0(t)$) and the other pair of external legs carry  $ x$  times the unit mass (corresponding to $ I_0(xt)$). If we further restrict the range of $x$ to the closed unit interval $[0,1]$, such a Feynman diagram is related to Kluyver's probability density function $p_5(x)$, as we demonstrate in the lemma below.

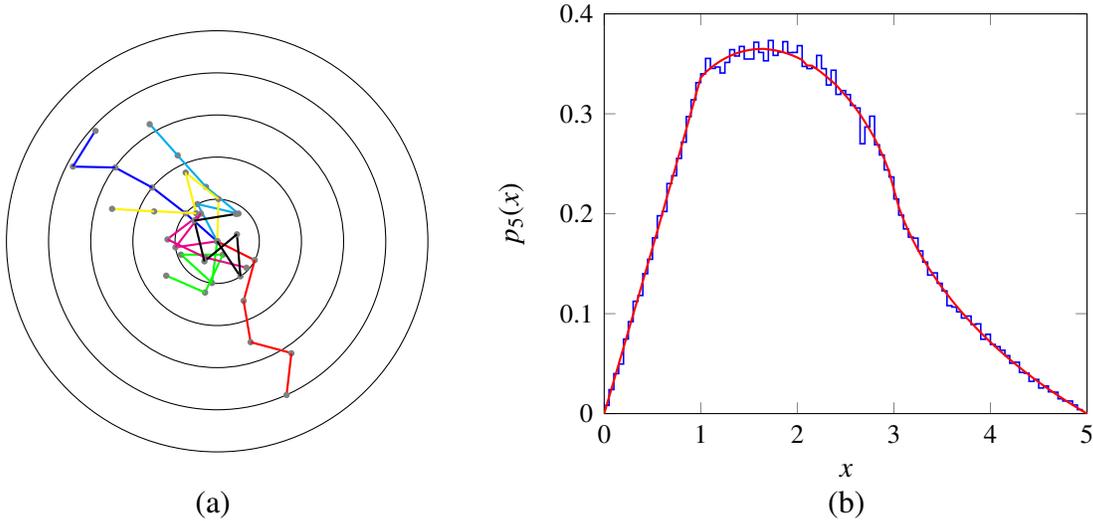
\begin{figure}[t]\begin{minipage}{0.38\textwidth}\begin{center}\begin{picture}(0,0)(0,0)
\put(0,0){\circle{400}}
\put(0,0){\circle{320}}
\put(0,0){\circle{240}}
\put(0,0){\circle{160}}
\put(0,0){\circle{80}}
\put(0,0){\color{gray}\circle*{6}}
\put(0,0){\thicklines\color{blue}{\line(-0.749744,0.661728){29.9898}}}
\put(-29.9898,26.4691){\color{gray}\circle*{6}}
\put(-29.9898,26.4691){\thicklines\color{blue}{\line(-0.790159,0.612902){31.6063}}}
\put(-61.5961,50.9852){\color{gray}\circle*{6}}
\put(-61.5961,50.9852){\thicklines\color{blue}{\line(-0.879822,0.475303){35.1929}}}
\put(-96.789,69.9974){\color{gray}\circle*{6}}
\put(-96.789,69.9974){\thicklines\color{blue}{\line(-0.999776,0.0211722){39.991}}}
\put(-136.78,70.8442){\color{gray}\circle*{6}}
\put(-136.78,70.8442){\thicklines\color{blue}{\line(0.530447,0.847718){21.2179}}}
\put(-115.562,104.753){\color{gray}\circle*{6}}
\put(0,0){\color{gray}\circle*{6}}
\put(0,0){\thicklines\color{green}{\line(-0.122889,-0.99242){4.91555}}}
\put(-4.91555,-39.6968){\color{gray}\circle*{6}}
\put(-4.91555,-39.6968){\thicklines\color{green}{\line(-0.738098,0.674693){29.5239}}}
\put(-34.4395,-12.7091){\color{gray}\circle*{6}}
\put(-34.4395,-12.7091){\thicklines\color{green}{\line(1.,0.00088566){40.}}}
\put(5.5605,-12.6737){\color{gray}\circle*{6}}
\put(5.5605,-12.6737){\thicklines\color{green}{\line(-0.429121,-0.903247){17.1648}}}
\put(-11.6043,-48.8036){\color{gray}\circle*{6}}
\put(-11.6043,-48.8036){\thicklines\color{green}{\line(-0.916822,0.399295){36.6729}}}
\put(-48.2772,-32.8317){\color{gray}\circle*{6}}
\put(0,0){\color{gray}\circle*{6}}
\put(0,0){\thicklines\color{red}{\line(0.893104,-0.44985){35.7242}}}
\put(35.7242,-17.994){\color{gray}\circle*{6}}
\put(35.7242,-17.994){\thicklines\color{red}{\line(-0.262647,-0.964892){10.5059}}}
\put(25.2183,-56.5897){\color{gray}\circle*{6}}
\put(25.2183,-56.5897){\thicklines\color{red}{\line(0.167226,-0.985919){6.68904}}}
\put(31.9073,-96.0264){\color{gray}\circle*{6}}
\put(31.9073,-96.0264){\thicklines\color{red}{\line(0.966282,-0.257486){38.6513}}}
\put(70.5586,-106.326){\color{gray}\circle*{6}}
\put(70.5586,-106.326){\thicklines\color{red}{\line(-0.113814,-0.993502){4.55255}}}
\put(66.0061,-146.066){\color{gray}\circle*{6}}
\put(0,0){\color{gray}\circle*{6}}
\put(0,0){\thicklines\color{cyan}{\line(-0.465476,0.885061){18.619}}}
\put(-18.619,35.4024){\color{gray}\circle*{6}}
\put(-18.619,35.4024){\thicklines\color{cyan}{\line(0.973969,-0.226682){38.9588}}}
\put(20.3397,26.3351){\color{gray}\circle*{6}}
\put(20.3397,26.3351){\thicklines\color{cyan}{\line(-0.772585,0.634911){30.9034}}}
\put(-10.5637,51.7316){\color{gray}\circle*{6}}
\put(-10.5637,51.7316){\thicklines\color{cyan}{\line(-0.670441,0.741963){26.8176}}}
\put(-37.3813,81.4101){\color{gray}\circle*{6}}
\put(-37.3813,81.4101){\thicklines\color{cyan}{\line(-0.665178,0.746685){26.6071}}}
\put(-63.9884,111.278){\color{gray}\circle*{6}}
\put(0,0){\color{gray}\circle*{6}}
\put(0,0){\thicklines\color{magenta}{\line(-0.989797,-0.142482){39.5919}}}
\put(-39.5919,-5.6993){\color{gray}\circle*{6}}
\put(-39.5919,-5.6993){\thicklines\color{magenta}{\line(0.606411,0.795151){24.2565}}}
\put(-15.3354,26.1067){\color{gray}\circle*{6}}
\put(-15.3354,26.1067){\thicklines\color{magenta}{\line(-0.795362,-0.606135){31.8145}}}
\put(-47.1499,1.86133){\color{gray}\circle*{6}}
\put(-47.1499,1.86133){\thicklines\color{magenta}{\line(0.901088,-0.433636){36.0435}}}
\put(-11.1064,-15.4841){\color{gray}\circle*{6}}
\put(-11.1064,-15.4841){\thicklines\color{magenta}{\line(0.970893,-0.239514){38.8357}}}
\put(27.7293,-25.0647){\color{gray}\circle*{6}}
\put(0,0){\color{gray}\circle*{6}}
\put(0,0){\thicklines\color{yellow}{\line(0.0320912,0.999485){1.28365}}}
\put(1.28365,39.9794){\color{gray}\circle*{6}}
\put(1.28365,39.9794){\thicklines\color{yellow}{\line(-0.777568,0.628798){31.1027}}}
\put(-29.8191,65.1313){\color{gray}\circle*{6}}
\put(-29.8191,65.1313){\thicklines\color{yellow}{\line(0.249103,-0.968477){9.96414}}}
\put(-19.8549,26.3923){\color{gray}\circle*{6}}
\put(-19.8549,26.3923){\thicklines\color{yellow}{\line(-0.998738,0.0502189){39.9495}}}
\put(-59.8045,28.401){\color{gray}\circle*{6}}
\put(-59.8045,28.401){\thicklines\color{yellow}{\line(-0.998174,0.060398){39.927}}}
\put(-99.7314,30.8169){\color{gray}\circle*{6}}
\put(0,0){\color{gray}\circle*{6}}
\put(0,0){\thicklines\color{black}{\line(0.549361,-0.835585){21.9744}}}
\put(21.9744,-33.4234){\color{gray}\circle*{6}}
\put(21.9744,-33.4234){\thicklines\color{black}{\line(-0.0816381,0.996662){3.26552}}}
\put(18.7089,6.44308){\color{gray}\circle*{6}}
\put(18.7089,6.44308){\thicklines\color{black}{\line(-0.768536,-0.639806){30.7414}}}
\put(-12.0325,-19.1492){\color{gray}\circle*{6}}
\put(-12.0325,-19.1492){\thicklines\color{black}{\line(-0.240049,0.970761){9.60198}}}
\put(-21.6345,19.6812){\color{gray}\circle*{6}}
\put(-21.6345,19.6812){\thicklines\color{black}{\line(0.987449,0.157939){39.498}}}
\put(17.8635,25.9988){\color{gray}\circle*{6}}
\end{picture}\end{center}\end{minipage}\begin{minipage}{0.5\textwidth}\begin{center}
\begin{tikzpicture}
\pgfplotsset{width=8cm,tick label style={font=\footnotesize},label style={font=\small}}
\begin{axis}[xlabel={$x$},ylabel={$ p_5(x)$},ymin=0,ymax=0.4,xmin=0,xmax=5,enlargelimits=false]
\addplot [
const plot,
draw=blue, semithick
] coordinates {
(0.,0.0082) (0.05,0.024) (0.1,0.04) (0.15,0.0496) (0.2,0.0746) (0.25,0.092) (0.3,0.1124) (0.35,0.1182) (0.4,0.14) (0.45,0.1544) (0.5,0.1758) (0.55,0.198) (0.6,0.2022) (0.65,0.2302) (0.7,0.238) (0.75,0.2554) (0.8,0.2716) (0.85,0.2972) (0.9,0.314) (0.95,0.3312) (1.,0.34) (1.05,0.3554) (1.1,0.3454) (1.15,0.3468) (1.2,0.3408) (1.25,0.3514) (1.3,0.3642) (1.35,0.3578) (1.4,0.3674) (1.45,0.3548) (1.5,0.3548) (1.55,0.3716) (1.6,0.3614) (1.65,0.353) (1.7,0.3734) (1.75,0.3584) (1.8,0.3618) (1.85,0.3722) (1.9,0.3612) (1.95,0.3614) (2.,0.3682) (2.05,0.3474) (2.1,0.3452) (2.15,0.3462) (2.2,0.3584) (2.25,0.3452) (2.3,0.3288) (2.35,0.3436) (2.4,0.3194) (2.45,0.3232) (2.5,0.3192) (2.55,0.3084) (2.6,0.3058) (2.65,0.27) (2.7,0.2868) (2.75,0.2976) (2.8,0.269) (2.85,0.2588) (2.9,0.2438) (2.95,0.2368) (3.,0.2148) (3.05,0.1986) (3.1,0.1848) (3.15,0.1766) (3.2,0.1756) (3.25,0.1628) (3.3,0.1516) (3.35,0.1386) (3.4,0.1408) (3.45,0.1304) (3.5,0.1228) (3.55,0.108) (3.6,0.1066) (3.65,0.1018) (3.7,0.0956) (3.75,0.0976) (3.8,0.0892) (3.85,0.0898) (3.9,0.0744) (3.95,0.0794) (4.,0.0692) (4.05,0.067) (4.1,0.0634) (4.15,0.058) (4.2,0.0504) (4.25,0.0514) (4.3,0.0412) (4.35,0.0404) (4.4,0.0322) (4.45,0.0342) (4.5,0.0256) (4.55,0.0276) (4.6,0.0218) (4.65,0.0214) (4.7,0.0148) (4.75,0.0126) (4.8,0.0126) (4.85,0.0086) (4.9,0.004) (4.95,0.0018)
}
;
\addplot [
draw=red, thick
] table {
x y
0 0
 0.05 0.01650
 0.1 0.03300
 0.15 0.04951
 0.2 0.06604
 0.25 0.08257
 0.3 0.09918
 0.35 0.1157
 0.4 0.1324
 0.45 0.1491
 0.5 0.1658
 0.55 0.1825
 0.6 0.1994
 0.65 0.2168
 0.7 0.2333
 0.75 0.2503
 0.8 0.2674
 0.85 0.2846
 0.9 0.3019
 0.95 0.3193
 1. 0.3361
 1.05 0.3415
 1.1 0.3456
 1.15 0.3493
 1.2 0.3526
 1.25 0.3554
 1.3 0.3578
 1.35 0.3598
 1.4 0.3615
 1.45 0.3628
 1.5 0.3638
 1.55 0.3645
 1.6 0.3648
 1.65 0.3650
 1.7 0.3644
 1.75 0.3638
 1.8 0.3629
 1.85 0.3617
 1.9 0.3602
 1.95 0.3584
 2. 0.3563
 2.05 0.3539
 2.1 0.3483
 2.15 0.3482
 2.2 0.3449
 2.25 0.3413
 2.3 0.3374
 2.35 0.3331
 2.4 0.3291
 2.45 0.3235
 2.5 0.3182
 2.55 0.3124
 2.6 0.3061
 2.65 0.2994
 2.7 0.2919
 2.75 0.2841
 2.8 0.2755
 2.85 0.2655
 2.9 0.2549
 2.95 0.2423
 3. 0.2244
 3.05 0.2069
 3.1 0.1940
 3.15 0.1829
 3.2 0.1729
 3.25 0.1638
 3.3 0.1553
 3.35 0.1473
 3.4 0.1399
 3.45 0.1327
 3.5 0.1260
 3.55 0.1196
 3.6 0.1134
 3.65 0.1075
 3.7 0.1018
 3.75 0.09631
 3.8 0.09102
 3.85 0.08590
 3.9 0.08096
 3.95 0.07617
 4. 0.07152
 4.05 0.06702
 4.1 0.06264
 4.15 0.05838
 4.2 0.05425
 4.25 0.05021
 4.3 0.04629
 4.35 0.04246
 4.4 0.03872
 4.45 0.03508
 4.5 0.03152
 4.55 0.02805
 4.6 0.02465
 4.65 0.02133
 4.7 0.01809
 4.75 0.01491
 4.8 0.01180
 4.85 0.008760
 4.9 0.005780
 4.95 0.002861
 5 0
 };
\end{axis}
\end{tikzpicture}\end{center}
\end{minipage}
\begin{minipage}{0.45\textwidth}\begin{center}(a)\end{center}\end{minipage}
\begin{minipage}{0.5\textwidth}\begin{center}(b)\end{center}\end{minipage}\caption{(a)~Sample trajectories of simulated 5-step planar uniform random walks. (b)~Histogram (shown in \textit{blue}, binned with Freedman--Diaconis rule \cite{FreedmanDiaconis1981}) for distance $x $ traveled by a  rambler, constructed from $10^5$ simulated trajectories of   5-step uniform random walks in the Euclidean plane, superimposed with Kluyver's probability density function $ p_5(x)$ (\textit{red} curve). Note the close resemblance of $ p_5(x),0\leq x\leq 1$ to a straight line segment (Pearson--Fettis phenomenon).}\label{fig:p5}\end{figure}

\begin{lemma}[Alternative integral representation for $ p_5$]\label{lm:p5_alt_int}For $ x\in[0,1]$, the following identity holds:\begin{align}
p_5(x)=\frac{30}{\pi^{4}}\int_0^\infty I_0(xt)I_{0}(t)[K_0(t)]^4 xt\D t.\label{eq:p5_IKM}
\end{align}\end{lemma}\begin{proof}Let  $ Y_0(t)=-\frac{2}{\pi}\int_0^\infty\cos(t\cosh u)\D u,t>0$ be the Bessel function of the second kind (also known as the Neumann function) and\ zeroth order.
Define the Hankel function of the first kind and zeroth order as $ H_0^{(1)}(\xi)=J_0(\xi)+iY_0(\xi)$ for $\xi>0$.

Extending the definition of Bessel functions to complex arguments, we have  $ J_0(it)=I_0(t)$ and $\frac{\pi i}{2}H_0^{(1)}(it)=K_0(t)$,  so long as $ t>0$.  For $ x\in[0,3]$, we  perform a \textit{Wick rotation} as follows:\begin{align}\begin{split}\left(\frac2\pi
\right)^4\int_0^\infty I_0(xt)I_{0}(t)[K_0(t)]^4 t\D t={}&-\int_{0}^{i\infty}J_0(xz)J_0(z)[H_{0}^{(1)}(z)]^4 z\D z\\={}&-\R\int_{0}^{\infty}J_0(xt)J_0(t)[H_{0}^{(1)}(t)]^4 t\D t.\end{split}\label{eq:IIKKKK_Wick}
\end{align}Here, we can deform the path of integration from the positive $\I z $-axis to the positive $\R z$-axis, thanks to Jordan's lemma applicable to  the asymptotic behavior \begin{align}
H_0^{(1)}(z)=\sqrt{\frac{2}{\pi z}}e^{i\left(z-\frac{\pi}{4}\right)}\left[1+O\left( \frac{1}{|z|} \right)\right]\quad \text{and}\quad
J_{0}(z)=\sqrt{\frac{2}{\pi z}}\frac{e^{i\left(z-\frac{\pi}{4}\right)}+e^{-i\left(z-\frac{\pi}{4}\right)}}{2}\left[1+O\left( \frac{1}{|z|} \right)\right]
\end{align} as $ |z|\to\infty,-\pi<\arg z<\pi$.

Spelling out $H_{0}^{(1)}(t)=J+iY $ with self-explanatory abbreviations for $ t>0$, we rewrite \eqref{eq:IIKKKK_Wick} as \begin{align}
\left(\frac2\pi
\right)^4\int_0^\infty I_0(xt)I_{0}(t)[K_0(t)]^4 t\D t=-\int_{0}^{\infty}J_0(xt)J(J^{4}-6J^{2}Y^{2}+Y^{4}) t\D t.\label{eq:IIKKKK_Wick'}
\end{align}Now, Jordan's lemma brings us two vanishing identities:\begin{align}
\int_{i0^+-\infty}^{i0^++\infty}J_0(xz)[H_0^{(1)}(z)]^5z\D z={}&0,&\forall x\in[0,5],\intertext{and}\int_{i0^+-\infty}^{i0^++\infty}J_0(xz)[J_{0}(z)]^{2}[H_0^{(1)}(z)]^3z\D z={}&0,& \forall x\in[0,1],\label{eq:Jordan_2}
\end{align} where the contours close upwards. Noting that  $ J_0(\xi)=J_0(-\xi)$, $ H_0^{(1)}(\pm\xi+i0^+)=\pm J_0(\xi)+iY_0(\xi)$ for $\xi>0$, and \begin{align}
\begin{split}&-J (J^4-6 J^2 Y^2+Y^4)\\{}&+\frac{(J+i Y)^5-(-J+i Y)^5}{10} +
\frac{2J^2}{3}  [(J+i Y)^3-(-J+i Y)^3]=\frac{8 J^5}{15},\end{split}\label{eq:Y_cancel}
\end{align}we simplify \eqref{eq:IIKKKK_Wick'} into\begin{align}
\left(\frac2\pi
\right)^4\int_0^\infty I_0(xt)I_{0}(t)[K_0(t)]^4 t\D t=\frac{8}{15}\int_{0}^{\infty}J_0(xt)J^{5} t\D t,
\end{align}as claimed in  \eqref{eq:p5_IKM}.
 \end{proof}

 The identity  in  \eqref{eq:p5_IKM} has some importance to  the quantitative understanding of the 3-loop sunrise diagram (according to Broadhurst's normalization \cite[(84)]{Broadhurst2016})\begin{align}
\;\;\;\;\;
\dia{\put(-50,0){\line(-1,0){50}}
\put(50,0){\line(1,0){50}}
\put(0,0){\circle{100}}
\qbezier(-50, 0)(0, 35)(50, 0)
\qbezier(-50, 0)(0, -35)(50, 0)
\put(50,0){\vtx}
\put(-50,0){\vtx}
}{}\;\;=2^{3}\int_0^\infty I_{0}(t)[K_0(t)]^4 t\D t
\end{align}  in 2-dimensional quantum field theory. The leading Taylor coefficient $ r_{5,0}^{\phantom'}=p'_5(0^{+})=\int_0^\infty [J_0(t)]^5t\D t$ had been evaluated in closed form \cite[Theorem 5.1]{BSWZ2012} before the first rigorous proof \cite{BlochKerrVanhove2015,Samart2016} of \begin{align}
\int_{0}^\infty I_0(t)[K_0(t)]^4t\D t=\frac{1}{240 \sqrt{5}}\Gamma \left(\frac{1}{15}\right) \Gamma \left(\frac{2}{15}\right) \Gamma \left(\frac{4}{15}\right) \Gamma \left(\frac{8}{15}\right)\label{eq:3loop_sunrise}
\end{align}was found, using sophisticated techniques in algebraic geometry and number theory. Now we see that  the evaluations in \eqref{eq:r50} and  \eqref{eq:3loop_sunrise} are equivalent to each other, up to contour deformations that we showed in \cite[Theorem 2.2.2]{Zhou2017WEF} and Lemma \ref{lm:p5_alt_int}  above.

 Unlike Kluyver's original integral representation for $ p_5(x)$, we can differentiate (with respect to $x$) under the integral sign in the Feynman diagram $ \int_0^\infty I_0(xt)I_{0}(t)[K_0(t)]^4 t\D t$, to obtain an integral representation for $ p'''_5(0^{+})$.
 With this convenient operation, we give a new proof of Borwein's conjectural identity  in \eqref{eq:B_conj}, without invoking the 4-dimensional uniform random walks studied by Borwein--Straub--Vignat \cite[Example 4.15]{BSV2016}.   \begin{theorem}[Borwein--Straub--Vignat]We have \begin{align}
r_{5,1}=\frac{15}{2\pi^{4}}\int_0^\infty I_{0}(t)[K_0(t)]^4 t^{3}\D t=\frac{13}{225}r_{5,0}-\frac{2}{5\pi^4r_{5,0}}.
\end{align}\end{theorem}\begin{proof}The first equality comes from third-order derivatives of \eqref{eq:p5_IKM} and the second equality has been verified in \cite[Theorem 3.3]{Zhou2017WEF}. \end{proof}
\section{Maclaurin expansions of Kluyver's probability density functions \label{sec:Maclaurin}}
Carrying the analysis in  \S\ref{sec:new_pf_B_conj} a little further, we obtain the following explicit formula for the Taylor coefficients as Bessel moments:\begin{align}
r_{5,k}=\frac{30}{4^{k}(k!)^2\pi^4}\int_0^\infty I_{0}(t)[K_0(t)]^4 t^{2k+1}\D t=:\frac{30s_{5,2k+1}}{4^{k}(k!)^2\pi^4}.\label{eq:r5k_int_repn}
\end{align} Therefore, we expect that the four-term recurrence relations for $ r_{5,k}$ \cite[(5.6)]{BSWZ2012} and $ s_{5,2k+1}$ \cite[(10)--(11)]{Bailey2008} to be compatible with each other. In particular, all these Taylor coefficients are recursively determined by the following initial conditions (cf.~conjectures in \cite[(95)--(97)]{Bailey2008} proved by  \cite[Theorem 3.3]{Zhou2017WEF})\begin{align}\left\{\begin{array}{r@{\;=\;}l@{\;=\;}l}r_{5,0}=\dfrac{30s_{5,1}}{\pi^4}&\dfrac{30C}{\pi^2}&0.3299338011...\\[12pt]r_{5,1}=\dfrac{15s_{5,3}}{2\pi^{4}}&\dfrac{2}{15\pi^2}\left( 13C-\dfrac{1}{10C} \right)&0.006616730259...\\[12pt]r_{5,2}=\dfrac{15s_{5,5}}{32\pi^{4}}&\dfrac{2}{225\pi^{2}}\left( 43C-\dfrac{19}{40C} \right)&0.0002623323540...\end{array}\right.
\label{eq:leading_coeff}\end{align} where $ C=\frac{1}{240 \sqrt{5}\pi^{2}}\Gamma \left(\frac{1}{15}\right) \Gamma \left(\frac{2}{15}\right) \Gamma \left(\frac{4}{15}\right) \Gamma \left(\frac{8}{15}\right)$ is the \textit{Bologna constant} attributed to Broadhurst \cite{Bailey2008} and Laporta \cite{Laporta2008}.

 Like the rescaled Bessel moments $ s_{5,2k+1}/\pi^2$, the rescaled Taylor coefficients $ r_{5,k}\pi^2$ are always \textit{strictly positive} numbers that are rational combinations of  the Bologna constant  $C$ and its reciprocal $1/C$. Here, we note that the strict positivity of the sequence  $ r_{5,k},k\in\mathbb Z_{\geq0}$ does not follow immediately from either the recurrence relation in \cite[(5.6)]{BSWZ2012} or Kluyver's integral representation for $ p_5(x)$ in \eqref{eq:Kluyver_pn}. Instead, it hinges on the Taylor expansion of $ I_0(xt)$ as well as  the positivity of $ I_0(t)$ and $K_0(t) $.

Numerical experiments led the authors of \cite{BSWZ2012} to the impression that the Maclaurin series $\sum^\infty_{k=0}r_{5,k}x^{2k+1}$ ``appears to converge for $ |x|<3$''. Borwein and coworkers further explained this as ``in accordance with $ \frac19$ being a root of the characteristic polynomial of the recurrence [for $ r_{5,k}$]''.
The following theorem offers an analytic  perspective on the convergence of this Maclaurin series.
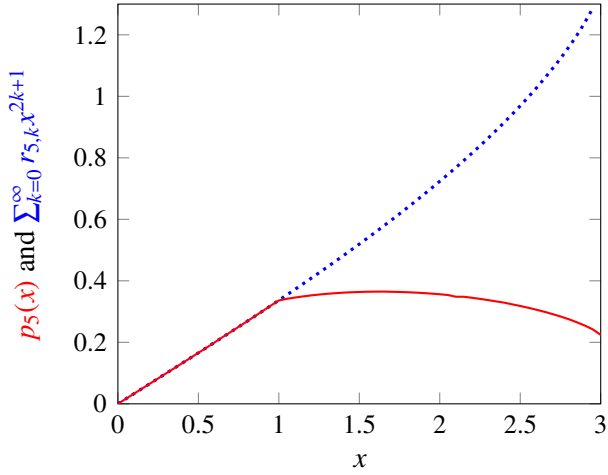
\begin{figure}[t]\begin{minipage}{0.5\textwidth}\begin{center}\begin{tikzpicture}
\pgfplotsset{width=8cm,tick label style={font=\footnotesize},label style={font=\small}}
\begin{axis}[xlabel={$x$},ylabel={$\textcolor{red}{ p_5(x)}$ and \textcolor{blue}{$ \sum^\infty_{k=0}r_{5,k}x^{2k+1} $}},ymin=0,ymax=1.3,xmin=0,xmax=3,enlargelimits=false]
\addplot [
smooth, dotted,
draw=blue, very thick
] table {
x y
0. 0
 0.05 0.01650
 0.1 0.03300
 0.15 0.04951
 0.2 0.06604
 0.25 0.08259
 0.3 0.09916
 0.35 0.1158
 0.4 0.1324
 0.45 0.1491
 0.5 0.1658
 0.55 0.1826
 0.6 0.1994
 0.65 0.2163
 0.7 0.2333
 0.75 0.2503
 0.8 0.2674
 0.85 0.2846
 0.9 0.3019
 0.95 0.3193
 1. 0.3368
 1.05 0.3544
 1.1 0.3722
 1.15 0.3901
 1.2 0.4081
 1.25 0.4262
 1.3 0.4445
 1.35 0.4630
 1.4 0.4816
 1.45 0.5005
 1.5 0.5195
 1.55 0.5387
 1.6 0.5582
 1.65 0.5779
 1.7 0.5978
 1.75 0.6180
 1.8 0.6385
 1.85 0.6593
 1.9 0.6804
 1.95 0.7019
 2. 0.7237
 2.05 0.7459
 2.1 0.7685
 2.15 0.7915
 2.2 0.8151
 2.25 0.8391
 2.3 0.8638
 2.35 0.8890
 2.4 0.9150
 2.45 0.9416
 2.5 0.9691
 2.55 0.9975
 2.6 1.027
 2.65 1.058
 2.7 1.090
 2.75 1.123
 2.8 1.159
 2.85 1.197
 2.9 1.239
 2.95 1.286
 }
;
\addplot [
draw=red, thick
] table {
x y
0 0
 0.05 0.01650
 0.1 0.03300
 0.15 0.04951
 0.2 0.06604
 0.25 0.08257
 0.3 0.09918
 0.35 0.1157
 0.4 0.1324
 0.45 0.1491
 0.5 0.1658
 0.55 0.1825
 0.6 0.1994
 0.65 0.2168
 0.7 0.2333
 0.75 0.2503
 0.8 0.2674
 0.85 0.2846
 0.9 0.3019
 0.95 0.3193
 1. 0.3361
 1.05 0.3415
 1.1 0.3456
 1.15 0.3493
 1.2 0.3526
 1.25 0.3554
 1.3 0.3578
 1.35 0.3598
 1.4 0.3615
 1.45 0.3628
 1.5 0.3638
 1.55 0.3645
 1.6 0.3648
 1.65 0.3650
 1.7 0.3644
 1.75 0.3638
 1.8 0.3629
 1.85 0.3617
 1.9 0.3602
 1.95 0.3584
 2. 0.3563
 2.05 0.3539
 2.1 0.3483
 2.15 0.3482
 2.2 0.3449
 2.25 0.3413
 2.3 0.3374
 2.35 0.3331
 2.4 0.3291
 2.45 0.3235
 2.5 0.3182
 2.55 0.3124
 2.6 0.3061
 2.65 0.2994
 2.7 0.2919
 2.75 0.2841
 2.8 0.2755
 2.85 0.2655
 2.9 0.2549
 2.95 0.2423
 3. 0.2244
 };
\end{axis}
\end{tikzpicture}\end{center}\end{minipage}\begin{minipage}{0.5\textwidth}
\caption{(Adapted from \cite[Fig.~3]{BSWZ2012}) Kluyver's 5-step density function $ p_5(x)$ (\textit{red solid}  curve) and the Maclaurin series $ \sum^\infty_{k=0}r_{5,k}x^{2k+1}=\frac{30}{\pi^{4}}\int_0^\infty I_0(xt)I_{0}(t)[K_0(t)]^4 xt\D t$ (\textit{blue dotted}  curve) for $0\leq x\leq 3 $.\label{fig:p5_Taylor}}\end{minipage}\end{figure}

 \begin{theorem}[Taylor expansion for $ p_5(x)$]\label{thm:p5_Taylor}The Maclaurin series\begin{align}
\sum^\infty_{k=0}r_{5,k}x^{2k+1}\label{eq:Taylor_p5}
\end{align}converges uniformly to a continuous function for $  x\in[-3, 3]$, and agrees with $ p_5(x)$ for $  x\in[0,1]$. (See Fig.~\ref{fig:p5_Taylor} for comparison.)\end{theorem}\begin{proof}First we recall the asymptotic behavior of modified Bessel functions as follows:\begin{align}
I_{0}(t)=\frac{e^{t}}{\sqrt{2\pi t}}\left[ 1+O\left( \frac{1}{t} \right) \right],\quad K_0(t)=\frac{\pi e^{-t}}{\sqrt{2\pi t}}\left[ 1+O\left( \frac{1}{t} \right) \right],
\end{align}where $t $ is large and positive. Thus, for each fixed  $ x\in[-3,3]$, the function $f_{x}(t)=I_0(xt)I_{0}(t)[K_0(t)]^4 t$, $t>0$ is Lebesgue integrable. Furthermore, we have $ |f_x(t)|<f_y(t)$ for all $ t>0$, if $ |x|<|y|$.

By   Levi's monotone convergence theorem (or Lebesgue's dominated convergence theorem), we have \begin{align}\begin{split}
\sum^n_{k=0}r_{5,k}x^{2k+1}={}&\frac{30}{\pi^{4}}\int_0^\infty\left[
\sum^n_{k=0}\frac{1}{(k!)^{2}}\left(\frac{xt}{2}\right)^{2k}\right]I_{0}(t)[K_0(t)]^4 xt\D t\\\to{}& \frac{30}{\pi^{4}}\int_0^\infty I_0(xt)I_{0}(t)[K_0(t)]^4 xt\D t,\quad \text{as }n\to\infty,\end{split}
\end{align}for each fixed $ x\in[-3,3]$. Appealing again to Lebesgue's dominated convergence theorem, we can verify that \begin{align}
\lim_{n\to\infty}\int_0^\infty I_0(x_{n}t)I_{0}(t)[K_0(t)]^4 t\D t=\int_0^\infty I_0(xt)I_{0}(t)[K_0(t)]^4 t\D t
\end{align} holds for any sequence $ \{x_n|n\in\mathbb Z_{>0}\}$ in $ [-3,3]$ that converges to a point $x\in[-3,3]$.
This shows that $ \int_0^\infty I_0(xt)I_{0}(t)[K_0(t)]^4 xt\D t$ defines a continuous function with respect to $ x\in[-3,3]$.

Since we have a monotone sequence of continuous functions $ S_n(x)=\sum^n_{k=0}r_{5,k}x^{2k+1}$ converging pointwise to a continuous function $S(x)=\frac{30}{\pi^{4}}\int_0^\infty I_0(xt)I_{0}(t)[K_0(t)]^4 xt\D t $ on a compact interval $ [0,3]$, the convergence is uniform (\textit{i.e.}~$\lim_{n\to\infty}\max_{x\in[0,3]}|S_n(x)-S(x)|=0$), according to Dini's theorem. The uniformity extends to $x\in[-3,3] $ by symmetry.

 Therefore, for $ x\in[-3,3]$, the Maclaurin series in question converges uniformly  to a continuous function  $S(x)= \frac{30}{\pi^{4}}\int_0^\infty I_0(xt)I_{0}(t)[K_0(t)]^4 xt\D t$, which  agrees with $ p_5(x)$ when $  x\in[0, 1]$, in view of \eqref{eq:p5_IKM}. \end{proof}

Without prior numerical knowledge of  the  Taylor coefficients in \eqref{eq:leading_coeff},  we can still inspect the integral representation of $ r_{5,k}$ in \eqref{eq:r5k_int_repn}   and conclude that the magnitude of $ r_{5,0}$ overwhelms all subleading terms. In the light of this, the probability density  $ p_5(x),0\leq x\leq 1$, when approximated by truncated Maclaurin series, is almost a straight line segment with slope $ r_{5,0}^{\phantom'}=p'_5(0^{+})$. This approximate linearity is called the \textit{Pearson--Fettis phenomenon}, as described by Karl Pearson  in 1906 \cite{Pearson}:

\begin{center}\begin{minipage}{0.8\textwidth}``the graphical construction (cf.~Fig.~\ref{fig:p5}\textit{b}), however carefully reinvestigated, did not permit
of our considering the curve to be anything but a straight line... Even if it is not
absolutely true, it exemplifies the extraordinary power of such integrals of $J$
products [\textit{i.e.} \eqref{eq:Kluyver_pn} divided by $x$, while setting $ n=5$] to give extremely close approximations to such simple
forms as horizontal lines.''\end{minipage}\end{center}\noindent With an arduous analysis of the Bessel function $ J_0$, Henry Fettis established the non-linearity of  $ p_5(x),0\leq x\leq 1$ in 1963 \cite{Fettis1963}.

Had $ p_5(x),0\leq x\leq 1$ been a straight line segment, we would expect an equality between $ p'_5(0^{+})$ and $p_5^{\phantom'}(1)$.
Numerically, the latter exceeds the former by a fraction of about 2\%. Even without actually computing these quantities, we can recover Fettis' result  on the non-linearity of  $ p_5(x),0\leq x\leq 1$, from a comparison of Bessel moments. \begin{theorem}[Fettis non-linearity]We have an inequality $ p'_5(0^{+})<p_5^{\phantom'}(1)$. More precisely, we have a lower bound\begin{align}
p_5^{\phantom'}(1)-p'_5(0^{+})>r_{5,1}^{\phantom'}=0.006616730259...
\end{align}\end{theorem}\begin{proof}From  the Wick rotation, we obtain\begin{align}
p'_5(0^+)=\frac{30}{\pi^{4}}\int_0^\infty I_{0}(t)[K_0(t)]^4 t\D t<\frac{30}{\pi^{4}}\int_0^\infty [I_{0}(t)]^{2}[K_0(t)]^4 t\D t=p_5^{\phantom'}(1),
\end{align}where the strict inequality descends from the elementary fact that  $ I_0(t)=\frac{1}{\pi}\int_0^\pi e^{t\cos\theta}\D\theta>\frac{1}{\pi}\int_0^\pi \D\theta=1$ for $ t>0$.
This
establishes the Fettis non-linearity in a qualitative manner.

From the convergent Taylor expansion studied in Theorem \ref{thm:p5_Taylor},  we have\begin{align}
p_5^{\phantom'}(1)-p'_5(0^+)=\sum^\infty_{k=1}r_{5,k}^{\phantom'}>r_{5,1}^{\phantom'}.
\end{align}It is worth noting that this lower bound estimate is already fairly close to the actual value of $ p_5^{\phantom'}(1)-p'_5(0^+)=0.006894160706...$ \end{proof}

Before stepping into the analysis of  $ p_n(x)$ for $ n>5$ in the next section, we briefly revisit some known results for $ p_3(x)$ and $p_4(x)
$ (see Fig.~\ref{fig:p3p4}), from the perspective of Feynman integrals in 2-dimensional quantum field theory.
\begin{figure}[hb]\begin{minipage}{0.38\textwidth}\begin{center}
\begin{tikzpicture}
\pgfplotsset{width=7cm,tick label style={font=\footnotesize},label style={font=\small}}
\begin{axis}[xlabel={$x$},ylabel={$ p_3(x)$},ymin=0,ymax=0.85,xmin=0,xmax=3,enlargelimits=false]
\addplot [
const plot,
draw=blue, semithick
] coordinates {
(0.,0.0102) (0.05,0.0258) (0.1,0.0436) (0.15,0.0706) (0.2,0.0844) (0.25,0.098) (0.3,0.133) (0.35,0.1476) (0.4,0.164) (0.45,0.189) (0.5,0.2214) (0.55,0.2418) (0.6,0.2642) (0.65,0.297) (0.7,0.34) (0.75,0.3736) (0.8,0.4338) (0.85,0.4644) (0.9,0.6042) (0.95,0.8342) (1.,0.8346) (1.05,0.625) (1.1,0.5634) (1.15,0.5252) (1.2,0.5002) (1.25,0.4668) (1.3,0.4614) (1.35,0.4152) (1.4,0.426) (1.45,0.4036) (1.5,0.3952) (1.55,0.3924) (1.6,0.3802) (1.65,0.3792) (1.7,0.3688) (1.75,0.3718) (1.8,0.3606) (1.85,0.3414) (1.9,0.3544) (1.95,0.3538) (2.,0.3438) (2.05,0.3348) (2.1,0.3414) (2.15,0.329) (2.2,0.2992) (2.25,0.3106) (2.3,0.309) (2.35,0.3076) (2.4,0.307) (2.45,0.304) (2.5,0.2966) (2.55,0.2944) (2.6,0.2936) (2.65,0.2836) (2.7,0.292) (2.75,0.2892) (2.8,0.2768) (2.85,0.2762) (2.9,0.2764) (2.95,0.2748)
}
;
\addplot [
draw=red, thick
] table {
x y
0. 0
 0.05 0.01839
 0.1 0.03688
 0.15 0.05555
 0.2 0.07451
 0.25 0.09387
 0.3 0.1138
 0.35 0.1343
 0.4 0.1556
 0.45 0.1780
 0.5 0.2017
 0.55 0.2269
 0.6 0.2541
 0.65 0.2839
 0.7 0.3170
 0.75 0.3548
 0.8 0.3993
 0.85 0.4542
 0.9 0.5280
 0.95 0.6472
 1. 15
 1.05 0.6843
 1.1 0.5918
 1.15 0.5408
 1.2 0.5063
 1.25 0.4806
 1.3 0.4603
 1.35 0.4436
 1.4 0.4295
 1.45 0.4173
 1.5 0.4066
 1.55 0.3970
 1.6 0.3884
 1.65 0.3806
 1.7 0.3734
 1.75 0.3668
 1.8 0.3606
 1.85 0.3549
 1.9 0.3495
 1.95 0.3444
 2. 0.3396
 2.05 0.3351
 2.1 0.3307
 2.15 0.3266
 2.2 0.3227
 2.25 0.3189
 2.3 0.3153
 2.35 0.3118
 2.4 0.3085
 2.45 0.3052
 2.5 0.3021
 2.55 0.2991
 2.6 0.2962
 2.65 0.2933
 2.7 0.2906
 2.75 0.2879
 2.8 0.2853
 2.85 0.2828
 2.9 0.2804
 2.95 0.2780
 3. 0.2757
  };
\end{axis}
\end{tikzpicture}\end{center}
\end{minipage}\begin{minipage}{0.5\textwidth}\begin{center}
\begin{tikzpicture}
\pgfplotsset{width=7cm,tick label style={font=\footnotesize},label style={font=\small}}
\begin{axis}[xlabel={$x$},ylabel={$ p_4(x)$},ymin=0,ymax=0.55,xmin=0,xmax=4,enlargelimits=false]
\addplot [
const plot,
draw=blue, semithick
] coordinates {
(0.,0.021) (0.05,0.0556) (0.1,0.0808) (0.15,0.1068) (0.2,0.119) (0.25,0.1284) (0.3,0.154) (0.35,0.1794) (0.4,0.188) (0.45,0.1954) (0.5,0.2232) (0.55,0.2382) (0.6,0.2536) (0.65,0.2556) (0.7,0.2716) (0.75,0.2782) (0.8,0.2914) (0.85,0.2916) (0.9,0.3204) (0.95,0.3176) (1.,0.3372) (1.05,0.33) (1.1,0.3516) (1.15,0.3748) (1.2,0.3728) (1.25,0.3886) (1.3,0.4008) (1.35,0.3994) (1.4,0.3964) (1.45,0.4158) (1.5,0.4228) (1.55,0.4242) (1.6,0.4498) (1.65,0.4556) (1.7,0.4794) (1.75,0.4582) (1.8,0.467) (1.85,0.486) (1.9,0.4886) (1.95,0.4972) (2.,0.4306) (2.05,0.403) (2.1,0.3616) (2.15,0.341) (2.2,0.344) (2.25,0.3154) (2.3,0.3078) (2.35,0.2876) (2.4,0.2612) (2.45,0.2564) (2.5,0.2624) (2.55,0.2442) (2.6,0.2444) (2.65,0.2348) (2.7,0.2226) (2.75,0.198) (2.8,0.217) (2.85,0.195) (2.9,0.189) (2.95,0.178) (3.,0.1744) (3.05,0.1594) (3.1,0.1662) (3.15,0.1514) (3.2,0.1534) (3.25,0.1388) (3.3,0.143) (3.35,0.1268) (3.4,0.1184) (3.45,0.1078) (3.5,0.1028) (3.55,0.0912) (3.6,0.0922) (3.65,0.085) (3.7,0.077) (3.75,0.0734) (3.8,0.0624) (3.85,0.0584) (3.9,0.0382) (3.95,0.0198)
}
;
\addplot [
draw=red, thick
] table {
x y
0. 0
 0.05 0.03857
 0.1 0.06663
 0.15 0.09076
 0.2 0.1124
 0.25 0.1321
 0.3 0.1504
 0.35 0.1674
 0.4 0.1835
 0.45 0.1987
 0.5 0.2132
 0.55 0.2269
 0.6 0.2401
 0.65 0.2528
 0.7 0.2649
 0.75 0.2767
 0.8 0.2880
 0.85 0.2990
 0.9 0.3096
 0.95 0.3199
 1. 0.3299
 1.05 0.3397
 1.1 0.3492
 1.15 0.3585
 1.2 0.3676
 1.25 0.3765
 1.3 0.3852
 1.35 0.3938
 1.4 0.4022
 1.45 0.4104
 1.5 0.4185
 1.55 0.4265
 1.6 0.4344
 1.65 0.4421
 1.7 0.4498
 1.75 0.4574
 1.8 0.4649
 1.85 0.4723
 1.9 0.4797
 1.95 0.4870
 2. 0.4942
 2.05 0.4103
 2.1 0.3789
 2.15 0.3559
 2.2 0.3372
 2.25 0.3211
 2.3 0.3069
 2.35 0.2940
 2.4 0.2822
 2.45 0.2713
 2.5 0.2610
 2.55 0.2513
 2.6 0.2421
 2.65 0.2334
 2.7 0.2250
 2.75 0.2168
 2.8 0.2090
 2.85 0.2014
 2.9 0.1940
 2.95 0.1867
 3. 0.1796
 3.05 0.1727
 3.1 0.1658
 3.15 0.1590
 3.2 0.1523
 3.25 0.1456
 3.3 0.1390
 3.35 0.1324
 3.4 0.1257
 3.45 0.1190
 3.5 0.1122
 3.55 0.1052
 3.6 0.09816
 3.65 0.09086
 3.7 0.08325
 3.75 0.07522
 3.8 0.06661
 3.85 0.05712
 3.9 0.04618
 3.95 0.03234
 4. 0
 };
\end{axis}
\end{tikzpicture}\end{center}
\end{minipage}
\begin{minipage}{0.45\textwidth}\begin{center}\hspace{-3.6em}(a)\end{center}\end{minipage}
\begin{minipage}{0.3\textwidth}\begin{center}(b)\end{center}\end{minipage}\caption{(a)~Histogram (\textit{blue}) constructed from  $10^5$   simulated 3-step planar uniform random walks, in comparison with  Kluyver's probability density function  (\textit{red}). Note that $ p_3(1^-)$ and $ p_3(1^+)$ both diverge. (b)~Analog of panel \textit{a} for 4-step planar uniform random walks. Note that the ``shark fin'' density function $ p_4(x)$ does not have a finite slope at the origin.}\label{fig:p3p4}\end{figure}
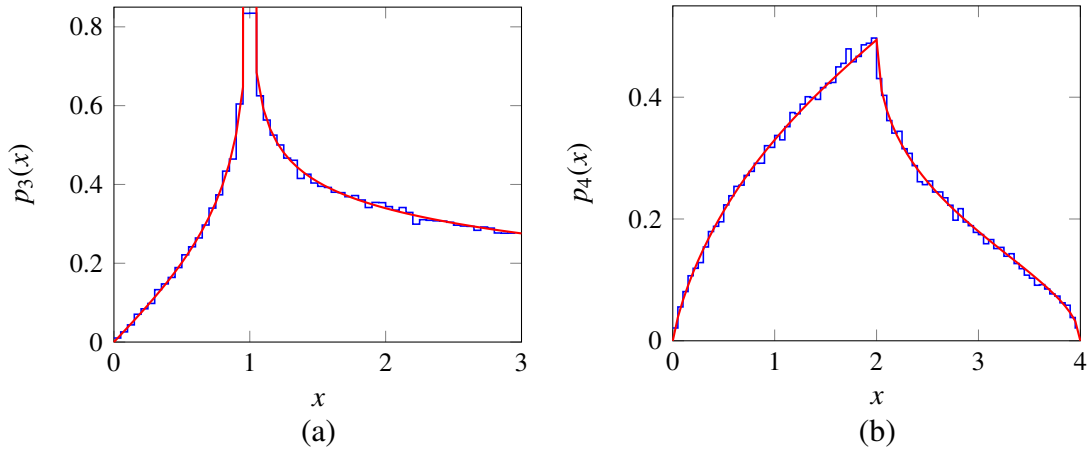

The 4-step density function $ p_4(x)$ does not admit a Maclaurin expansion, but has asymptotic behavior $ O(x\log x)$ for $ x\to0^+$  \cite[Theorem 4.4]{BSWZ2012}. To look at this through the lens of Bessel moments, we assemble the following formula from \cite[Propositions 3.1.2 and 5.1.4]{Zhou2017WEF}:\begin{align}\begin{split}
p_{4}(x):={}&\int_0^\infty J_0(xt)[J_0(t)]^4xt\D t\\={}&\frac{6}{\pi^4}\int_0^\infty I_0(xt)[K_0(t)]^4xt\D t+\frac{24}{\pi^4}\int_0^\infty K_0(xt)I_0(t)[K_0(t)]^3xt\D t, \end{split}
\end{align}where $x\in(0,2) $.  One can verify the integral identity above  by showing that both Kluyver's integral representation and the linear combination of Feynman diagrams satisfy the same homogeneous Picard--Fuchs differential equation, along with the same logarithmic asymptotic behavior as $ x\to0^+$. (An alternative derivation, based on Wick rotations and applications of Jordan's lemma, is left to interested readers.) We can recover the full asymptotic expansion of $ p_4(x)$ in \cite[Theorem 4.4]{BSWZ2012} via (generalized) power series  of $ I_0(xt)$ and $K_0(xt) $ around the origin.

The 3-step density function $ p_3(x)$ has a well-established Maclaurin expansion \cite[(3.2)]{BSWZ2012}, with strictly positive Taylor coefficients:
\begin{align}
p_{3}(x)=\frac{2x}{\pi\sqrt{3}}\sum^\infty_{k=0}\left[ \sum^k_{j=0} {k\choose j}^2{2j\choose j}\right]\left( \frac{x}{3} \right)^{2k},\quad 0\leq x<1
\end{align} where $ {n\choose k}=\frac{n!}{k!(n-k)!}$. Here, the same combinatorial coefficients also show up in the study of Bessel moments \cite[(23)--(24)]{Bailey2008}: \begin{align}
s_{3,2k+1}:=\int_0^\infty I_{0}(t)[K_{0}(t)]^{2}t^{2k+1}\D t=\frac{\pi}{3\sqrt{3}}\left( \frac{2^kk!}{3^{k}} \right)^2\sum^k_{j=0} {k\choose j}^2{2j\choose j}.
\end{align}This is not accidental. The underlying mechanism is the following integral identity \cite[Lemma 4.1.1]{Zhou2017WEF}:\begin{align}
p_{3}(x):={}&\int_0^\infty J_0(xt)[J_0(t)]^3xt\D t=\frac{6}{\pi^{2}}\int_0^\infty I_{0}(xt)I_{0}(t)[K_{0}(t)]^{2}xt\D t,\quad 0\leq x<1,\label{eq:p3_IKM}
\end{align}which is provable by Wick rotation. In this aspect,  $ p_3(x)$ and $ p_5(x)$ have something in common [cf.~\eqref{eq:p5_IKM} and \eqref{eq:p3_IKM}]. However, unlike the Maclaurin expansion of $ p_5(x)$, the leading Taylor coefficient for $ p_3(x)$ is not predominantly large. Moreover, $ p_5(1)$ is finite while $ p_3(1^-)$ diverges. These effects, when compounded, make it impossible for $ p_3(x),0\leq x<1$ to exhibit anything close to approximate linearity, as in the pronounced Pearson--Fettis phenomenon for $ p_5(x),0\leq x\leq 1$. \section{Some results for 6-step, 7-step and 8-step uniform random walks\label{sec:6_8}}
A general question for Kluyver's $ n$-step  density function $ p_n(x)$ is its  asymptotic expansion around the origin. For $ n>5$, not much is understood about the number theory behind these density functions, which contrasts with  well-established Rayleigh's approximations $ p_n(x)\sim\frac{2x}{n}e^{-x^2/n}$ \cite{Rayleigh1905} for $ n\to\infty$ (see Fig.~\ref{fig:p6p7p8}). In this section, we give an account for the latest progress on the leading asymptotic behavior for  $ p_6(x)$  and $ p_8(x)$, as well as the Maclaurin expansion for $ p_7(x)$.
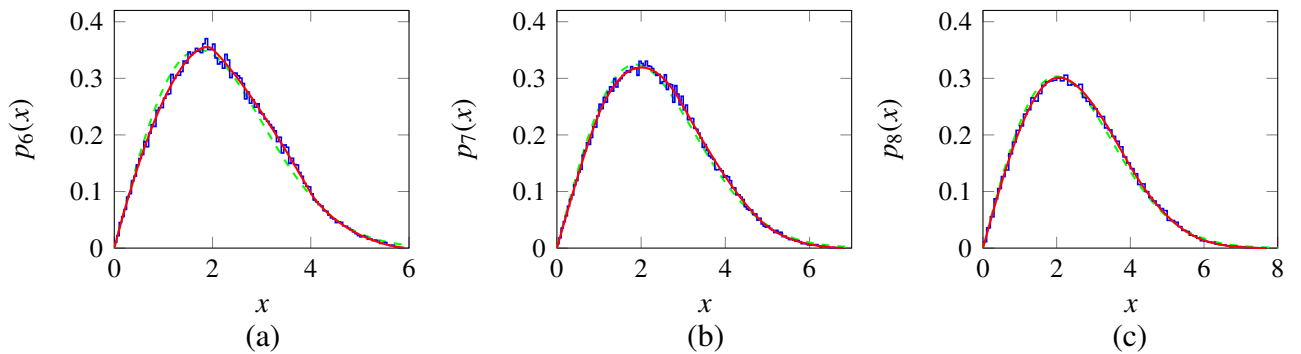
\begin{figure}[b]\begin{minipage}{0.33\textwidth}\begin{center}
\begin{tikzpicture}
\pgfplotsset{width=5.5cm,tick label style={font=\footnotesize},label style={font=\small}}
\begin{axis}[xlabel={$x$},ylabel={$ p_6(x)$},ymin=0,ymax=0.42,xmin=0,xmax=6,enlargelimits=false]
\addplot[green, dashed,thick,domain=0:6,samples=101] {(x/3)*exp(-x*x/6)};\addplot [
const plot,
draw=blue, semithick
] coordinates {(0.,0.0104) (0.05,0.022) (0.1,0.0452) (0.15,0.0556) (0.2,0.0684) (0.25,0.0902) (0.3,0.0962) (0.35,0.1186) (0.4,0.1272) (0.45,0.147) (0.5,0.1584) (0.55,0.1652) (0.6,0.1882) (0.65,0.179) (0.7,0.203) (0.75,0.2182) (0.8,0.2148) (0.85,0.2434) (0.9,0.2482) (0.95,0.2482) (1.,0.2648) (1.05,0.2714) (1.1,0.2728) (1.15,0.307) (1.2,0.2998) (1.25,0.3054) (1.3,0.3054) (1.35,0.312) (1.4,0.3306) (1.45,0.327) (1.5,0.3466) (1.55,0.342) (1.6,0.3458) (1.65,0.353) (1.7,0.3478) (1.75,0.3454) (1.8,0.3612) (1.85,0.3702) (1.9,0.3506) (1.95,0.3506) (2.,0.3606) (2.05,0.336) (2.1,0.3252) (2.15,0.329) (2.2,0.3178) (2.25,0.3424) (2.3,0.3324) (2.35,0.3012) (2.4,0.3086) (2.45,0.3094) (2.5,0.3046) (2.55,0.3004) (2.6,0.291) (2.65,0.2636) (2.7,0.2758) (2.75,0.256) (2.8,0.2622) (2.85,0.249) (2.9,0.255) (2.95,0.239) (3.,0.237) (3.05,0.2282) (3.1,0.2242) (3.15,0.2184) (3.2,0.2134) (3.25,0.2072) (3.3,0.1882) (3.35,0.196) (3.4,0.168) (3.45,0.1824) (3.5,0.1782) (3.55,0.1496) (3.6,0.1588) (3.65,0.1478) (3.7,0.147) (3.75,0.1336) (3.8,0.127) (3.85,0.1144) (3.9,0.1094) (3.95,0.1082) (4.,0.096) (4.05,0.085) (4.1,0.0818) (4.15,0.077) (4.2,0.0744) (4.25,0.0662) (4.3,0.0652) (4.35,0.058) (4.4,0.0534) (4.45,0.0514) (4.5,0.045) (4.55,0.0462) (4.6,0.0456) (4.65,0.0394) (4.7,0.0386) (4.75,0.035) (4.8,0.032) (4.85,0.0286) (4.9,0.0254) (4.95,0.022) (5.,0.0198) (5.05,0.0218) (5.1,0.0194) (5.15,0.0156) (5.2,0.016) (5.25,0.0134) (5.3,0.0142) (5.35,0.012) (5.4,0.0082) (5.45,0.0084) (5.5,0.01) (5.55,0.0058) (5.6,0.0044) (5.65,0.0052) (5.7,0.0024) (5.75,0.0022) (5.8,0.0008) (5.85,0.0006) (5.9,0.0006) }
;
\addplot [
draw=red, thick
] table {
x y
0. 0
 0.05 0.01664
 0.1 0.03288
 0.15 0.04873
 0.2 0.06420
 0.25 0.07924
 0.3 0.09396
 0.35 0.1083
 0.4 0.1222
 0.45 0.1358
 0.5 0.1490
 0.55 0.1619
 0.6 0.1743
 0.65 0.1864
 0.7 0.1982
 0.75 0.2096
 0.8 0.2206
 0.85 0.2312
 0.9 0.2415
 0.95 0.2514
 1. 0.2608
 1.05 0.2700
 1.1 0.2784
 1.15 0.2870
 1.2 0.2950
 1.25 0.3025
 1.3 0.3096
 1.35 0.3163
 1.4 0.3225
 1.45 0.3283
 1.5 0.3337
 1.55 0.3385
 1.6 0.3428
 1.65 0.3465
 1.7 0.3496
 1.75 0.3522
 1.8 0.3540
 1.85 0.3551
 1.9 0.3551
 1.95 0.3544
 2. 0.3519
 2.05 0.3474
 2.1 0.3428
 2.15 0.3381
 2.2 0.3332
 2.25 0.3282
 2.3 0.3230
 2.35 0.3177
 2.4 0.3123
 2.45 0.3068
 2.5 0.3012
 2.55 0.2954
 2.6 0.2896
 2.65 0.2836
 2.7 0.2776
 2.75 0.2715
 2.8 0.2655
 2.85 0.2590
 2.9 0.2527
 2.95 0.2462
 3. 0.2397
 3.05 0.2332
 3.1 0.2265
 3.15 0.2198
 3.2 0.2130
 3.25 0.2062
 3.3 0.1993
 3.35 0.1924
 3.4 0.1854
 3.45 0.1783
 3.5 0.1712
 3.55 0.1640
 3.6 0.1568
 3.65 0.1496
 3.7 0.1424
 3.75 0.1349
 3.8 0.1275
 3.85 0.1201
 3.9 0.1126
 3.95 0.1050
 4. 0.09744
 4.05 0.09082
 4.1 0.08500
 4.15 0.07967
 4.2 0.07472
 4.25 0.07009
 4.3 0.06573
 4.35 0.06163
 4.4 0.05775
 4.45 0.05406
 4.5 0.05056
 4.55 0.04724
 4.6 0.04411
 4.65 0.04107
 4.7 0.03820
 4.75 0.03546
 4.8 0.03286
 4.85 0.03038
 4.9 0.02801
 4.95 0.02575
 5. 0.02327
 5.05 0.02156
 5.1 0.01962
 5.15 0.01777
 5.2 0.01602
 5.25 0.01436
 5.3 0.01278
 5.35 0.01130
 5.4 0.009898
 5.45 0.008583
 5.5 0.007352
 5.55 0.006205
 5.6 0.005140
 5.65 0.004160
 5.7 0.003264
 5.75 0.002458
 5.8 0.001738
 5.85 0.001117
 5.9 0.0006013

  };
\end{axis}
\end{tikzpicture}\end{center}
\end{minipage}\begin{minipage}{0.33\textwidth}\begin{center}\begin{tikzpicture}
\pgfplotsset{width=5.5cm,tick label style={font=\footnotesize},label style={font=\small}}
\begin{axis}[xlabel={$x$},ylabel={$ p_7(x)$},ymin=0,ymax=0.42,xmin=0,xmax=7,enlargelimits=false]
\addplot[dashed,green,thick,domain=0:7,samples=101] {(2*x/7)*exp(-x*x/7)};\addplot [
const plot,
draw=blue, semithick
] coordinates {(0.,0.0078) (0.05,0.0182) (0.1,0.0354) (0.15,0.0468) (0.2,0.057) (0.25,0.0742) (0.3,0.0862) (0.35,0.0938) (0.4,0.1116) (0.45,0.1194) (0.5,0.134) (0.55,0.139) (0.6,0.1642) (0.65,0.172) (0.7,0.1864) (0.75,0.1832) (0.8,0.1954) (0.85,0.2142) (0.9,0.2138) (0.95,0.2336) (1.,0.2544) (1.05,0.2464) (1.1,0.2664) (1.15,0.2578) (1.2,0.2756) (1.25,0.2852) (1.3,0.2774) (1.35,0.2996) (1.4,0.2846) (1.45,0.2956) (1.5,0.2994) (1.55,0.305) (1.6,0.3102) (1.65,0.3136) (1.7,0.311) (1.75,0.3208) (1.8,0.3166) (1.85,0.3182) (1.9,0.3052) (1.95,0.3302) (2.,0.3264) (2.05,0.3224) (2.1,0.3304) (2.15,0.3218) (2.2,0.321) (2.25,0.3174) (2.3,0.3108) (2.35,0.3052) (2.4,0.307) (2.45,0.291) (2.5,0.2942) (2.55,0.3064) (2.6,0.2802) (2.65,0.29) (2.7,0.2886) (2.75,0.2576) (2.8,0.2872) (2.85,0.2722) (2.9,0.2532) (2.95,0.2684) (3.,0.2442) (3.05,0.2528) (3.1,0.2268) (3.15,0.2292) (3.2,0.229) (3.25,0.2116) (3.3,0.21) (3.35,0.2032) (3.4,0.1944) (3.45,0.1778) (3.5,0.183) (3.55,0.1726) (3.6,0.1628) (3.65,0.1596) (3.7,0.1568) (3.75,0.1386) (3.8,0.1384) (3.85,0.1396) (3.9,0.1366) (3.95,0.1252) (4.,0.1278) (4.05,0.126) (4.1,0.1192) (4.15,0.1066) (4.2,0.1064) (4.25,0.094) (4.3,0.0942) (4.35,0.0902) (4.4,0.086) (4.45,0.0732) (4.5,0.069) (4.55,0.0668) (4.6,0.0674) (4.65,0.0592) (4.7,0.0604) (4.75,0.0494) (4.8,0.0528) (4.85,0.0426) (4.9,0.0388) (4.95,0.037) (5.,0.0388) (5.05,0.0348) (5.1,0.0318) (5.15,0.0274) (5.2,0.0314) (5.25,0.0264) (5.3,0.028) (5.35,0.0202) (5.4,0.0212) (5.45,0.0178) (5.5,0.015) (5.55,0.0156) (5.6,0.0144) (5.65,0.0124) (5.7,0.0136) (5.75,0.0114) (5.8,0.0078) (5.85,0.0086) (5.9,0.0072) (5.95,0.0056) (6.,0.0056) (6.05,0.0054) (6.1,0.0062) (6.15,0.0062) (6.2,0.0034) (6.25,0.002) (6.3,0.0034) (6.35,0.0024) (6.4,0.0014) (6.45,0.001) (6.5,0.0006) (6.55,0.0014) (6.6,0.0006) (6.65,0.0004) (6.7,0.0002) (6.75,0.) (6.8,0.0004) }
;
\addplot [
draw=red, thick
] table {
x y
0. 0
 0.05 0.01304
 0.1 0.02607
 0.15 0.03904
 0.2 0.05201
 0.25 0.06490
 0.3 0.07760
 0.35 0.09042
 0.4 0.1031
 0.45 0.1155
 0.5 0.1279
 0.55 0.1401
 0.6 0.1521
 0.65 0.1639
 0.7 0.1755
 0.75 0.1867
 0.8 0.1979
 0.85 0.2083
 0.9 0.2186
 0.95 0.2284
 1. 0.2377
 1.05 0.2463
 1.1 0.2542
 1.15 0.2615
 1.2 0.2684
 1.25 0.2748
 1.3 0.2808
 1.35 0.2862
 1.4 0.2912
 1.45 0.2957
 1.5 0.2999
 1.55 0.3036
 1.6 0.3069
 1.65 0.3098
 1.7 0.3123
 1.75 0.3144
 1.8 0.3161
 1.85 0.3175
 1.9 0.3184
 1.95 0.3190
 2. 0.3192
 2.05 0.3191
 2.1 0.3186
 2.15 0.3177
 2.2 0.3165
 2.25 0.3149
 2.3 0.3130
 2.35 0.3107
 2.4 0.3081
 2.45 0.3052
 2.5 0.3019
 2.55 0.2982
 2.6 0.2943
 2.65 0.2899
 2.7 0.2852
 2.75 0.2804
 2.8 0.2751
 2.85 0.2694
 2.9 0.2635
 2.95 0.2572
 3. 0.2506
 3.05 0.2438
 3.1 0.2371
 3.15 0.2305
 3.2 0.2238
 3.25 0.2173
 3.3 0.2108
 3.35 0.2043
 3.4 0.1979
 3.45 0.1915
 3.5 0.1851
 3.55 0.1790
 3.6 0.1729
 3.65 0.1669
 3.7 0.1608
 3.75 0.1549
 3.8 0.1490
 3.85 0.1433
 3.9 0.1376
 3.95 0.1320
 4. 0.1264
 4.05 0.1210
 4.1 0.1156
 4.15 0.1104
 4.2 0.1052
 4.25 0.1001
 4.3 0.09514
 4.35 0.09027
 4.4 0.08549
 4.45 0.08083
 4.5 0.07630
 4.55 0.07185
 4.6 0.06753
 4.65 0.06334
 4.7 0.05928
 4.75 0.05536
 4.8 0.05158
 4.85 0.04796
 4.9 0.04450
 4.95 0.04124
 5. 0.03819
 5.05 0.03544
 5.1 0.03283
 5.15 0.03045
 5.2 0.02824
 5.25 0.02628
 5.3 0.02434
 5.35 0.02252
 5.4 0.02081
 5.45 0.01920
 5.5 0.01769
 5.55 0.01627
 5.6 0.01492
 5.65 0.01366
 5.7 0.01248
 5.75 0.01137
 5.8 0.01032
 5.85 0.009345
 5.9 0.008430
 5.95 0.007575
 6. 0.006777
 6.05 0.006034
 6.1 0.005341
 6.15 0.004705
 6.2 0.004114
 6.25 0.003570
 6.3 0.003071
 6.35 0.002616
 6.4 0.002201
 6.45 0.001828
 6.5 0.001492
 6.55 0.001195
 6.6 0.0009329
 6.65 0.0007061
 6.7 0.0005128
 6.75 0.0003521
 6.8 0.0002228
   };
\end{axis}
\end{tikzpicture}\end{center}
\end{minipage}\begin{minipage}{0.33\textwidth}\begin{center}
\begin{tikzpicture}
\pgfplotsset{width=5.5cm,tick label style={font=\footnotesize},label style={font=\small}}
\begin{axis}[xlabel={$x$},ylabel={$ p_8(x)$},ymin=0,ymax=0.42,xmin=0,xmax=8,enlargelimits=false]
\addplot[dashed,green,thick,domain=0:8,samples=101] {(x/4)*exp(-x*x/8)};\addplot [
const plot,
draw=blue, semithick
] coordinates {
(0.,0.0112) (0.1,0.0358) (0.2,0.0552) (0.3,0.0872) (0.4,0.1053) (0.5,0.1264) (0.6,0.138) (0.7,0.1666) (0.8,0.1873) (0.9,0.2081) (1.,0.2159) (1.1,0.2368) (1.2,0.2445) (1.3,0.2654) (1.4,0.2597) (1.5,0.273) (1.6,0.2875) (1.7,0.2939) (1.8,0.2953) (1.9,0.2965) (2.,0.3004) (2.1,0.2956) (2.2,0.3056) (2.3,0.288) (2.4,0.2894) (2.5,0.2868) (2.6,0.2894) (2.7,0.2661) (2.8,0.2593) (2.9,0.2536) (3.,0.2462) (3.1,0.2322) (3.2,0.2326) (3.3,0.2122) (3.4,0.209) (3.5,0.1975) (3.6,0.1889) (3.7,0.179) (3.8,0.1611) (3.9,0.1504) (4.,0.1405) (4.1,0.1314) (4.2,0.1123) (4.3,0.1135) (4.4,0.0998) (4.5,0.0893) (4.6,0.0833) (4.7,0.0761) (4.8,0.0699) (4.9,0.0663) (5.,0.049) (5.1,0.0488) (5.2,0.0457) (5.3,0.038) (5.4,0.0312) (5.5,0.0293) (5.6,0.0227) (5.7,0.0216) (5.8,0.0173) (5.9,0.013) (6.,0.012) (6.1,0.0106) (6.2,0.0102) (6.3,0.0069) (6.4,0.006) (6.5,0.0054) (6.6,0.0041) (6.7,0.0024) (6.8,0.0031) (6.9,0.0028) (7.,0.0024) (7.1,0.0008) (7.2,0.0003) (7.3,0.0007) (7.4,0.0002) (7.5,0.0001) (7.6,0.) (7.7,0.0001) }
;
\addplot [
draw=red, thick
] table {
x y
0. 0
 0.1 0.02372
 0.2 0.04718
 0.3 0.07020
 0.4 0.09284
 0.5 0.1144
 0.6 0.1354
 0.7 0.1554
 0.8 0.1744
 0.9 0.1925
 1. 0.2094
 1.1 0.2250
 1.2 0.2394
 1.3 0.2525
 1.4 0.2641
 1.5 0.2743
 1.6 0.2830
 1.7 0.2901
 1.8 0.2955
 1.9 0.2992
 2. 0.3011
 2.1 0.3014
 2.2 0.3002
 2.3 0.2977
 2.4 0.2940
 2.5 0.2892
 2.6 0.2835
 2.7 0.2770
 2.8 0.2696
 2.9 0.2615
 3. 0.2527
 3.1 0.2433
 3.2 0.2334
 3.3 0.2231
 3.4 0.2128
 3.5 0.2014
 3.6 0.1901
 3.7 0.1787
 3.8 0.1671
 3.9 0.1557
 4. 0.1445
 4.1 0.1336
 4.2 0.1232
 4.3 0.1133
 4.4 0.1035
 4.5 0.09491
 4.6 0.08640
 4.7 0.07836
 4.8 0.07078
 4.9 0.06365
 5. 0.05693
 5.1 0.05071
 5.2 0.04490
 5.3 0.03953
 5.4 0.03458
 5.5 0.03005
 5.6 0.02594
 5.7 0.02225
 5.8 0.01896
 5.9 0.01608
 6. 0.01360
 6.1 0.01150
 6.2 0.009676
 6.3 0.008099
 6.4 0.006732
 6.5 0.005548
 6.6 0.004528
 6.7 0.003652
 6.8 0.002904
 6.9 0.002272
 7. 0.001742
 7.1 0.001304
 7.2 0.0009463
 7.3 0.0006608
 7.4 0.0004385
 7.5 0.0002713
 7.6 0.0001517
 7.7 0.00007218
  };
\end{axis}
\end{tikzpicture}\end{center}
\end{minipage}
\begin{minipage}{0.35\textwidth}\begin{center}\hspace{-3.2em}(a)\end{center}\end{minipage}
\begin{minipage}{0.2\textwidth}\begin{center}\hspace{1.4em}(b)\end{center}\end{minipage}\begin{minipage}{0.3\textwidth}\begin{center}\hspace{7.4em}(c)\end{center}\end{minipage}\caption{Simulated histograms (\textit{blue}),    Kluyver's probability density functions  (\textit{red}), and Rayleigh's approximations $p_n(x)\sim \frac{2x}{n}e^{-x^2/n}$ (\textit{green}, \textit{dashed}) for $n\in\{6,7,8\}$.}\label{fig:p6p7p8}\end{figure}

Let   $ \Gamma(s)=\int_0^\infty t^{s-1}e^{-t}\D t, \R s>0$ be Euler's gamma function, with analytic continuations to all  $ s\in\mathbb C\smallsetminus\mathbb Z_{\leq0}$. In what follows, we introduce the  $L$-function associated with a cusp form $f$ via a Mellin transform:  \begin{align}
L(f,s):=\frac{(2\pi )^{s}}{\Gamma(s)}\int_0^{\infty} f(iy)y^{s-1}\D y.
\end{align}
Define the  Dedekind eta function as $ \eta(z):=e^{\pi iz/12}\prod_{n=1}^\infty(1-e^{2\pi inz})$ for complex numbers $z$ satisfying $ \I z>0$. We will be interested in the following three special cusp forms:\begin{align}
f_{3,15}(z)={}&[\eta(3z)\eta(5z)]^3+[\eta(z)\eta(15z)]^3,\\f_{4,6}(z)={}&[\eta(z)\eta(2z)\eta(3z)\eta(6z)]^{2},\\f_{6,6}(z)={}&
\frac{ [\eta (2 z) \eta (3 z)]^{9}}{[\eta (z)\eta (6 z)]^{3}}+\frac{ [\eta ( z) \eta (6 z)]^{9}}{[\eta (2z)\eta (3 z)]^{3}},\end{align}
where $f_{w,\ell} $ denotes a modular form of weight $w$ and level $\ell$.

At the time of writing (Aug.\ 2017), the derivatives $ p_5'(0^+)$, $ p'_6(0^+)$ and $ p'_8(0^+)$ are known to be representable by certain critical $L$-values. Here, a special $L$-value $L(f_{w,\ell},s)$ is said to be \textit{critical}, if $ s\in\mathbb Z\cap(0,w)$. We recapitulate these non-trivial evaluations from recent literature in the theorem below.\begin{theorem}[$p'_n(0^{+})$ and critical $L$-values]\label{thm:pn_L}We have \begin{align}
p_5'(0^{+})={}&\frac{6}{\pi^2}L(f_{3,15},1)=\frac{3 \sqrt{15}}{\pi ^3}L(f_{3,15},2),\label{eq:p5'_L}\\p'_6(0^{+})={}&\frac{15}{\pi^2}L(f_{4,6},1)=\frac{45}{\pi ^4}L(f_{4,6},3),\label{eq:p6'_L}\\p'_8(0^{+})={}&\frac{35}{9 \pi ^2}L(f_{6,6},1)=\frac{20}{\pi ^4}L(f_{6,6},3)=\frac{210}{\pi ^6}L(f_{6,6},5).\label{eq:p8'_L}
\end{align}\end{theorem}\begin{proof}
According to a result of Rogers--Wan--Zucker \cite[Theorem 5]{RogersWanZucker2015}, we have  $ L(f_{3,15},2)=\frac{r_{5,0}\pi^3}{3 \sqrt{15}}$. The relation between $ L(f_{3,15},1)$ and $ L(f_{3,15},2)$ is a consequence of the reflection formula for $ L(f_{3,15},s)$ \cite[(95)]{Broadhurst2016}. This proves \eqref{eq:p5'_L}.

 Setting $ x=1$ in \eqref{eq:p5_IKM}, we obtain\begin{align}
p'_6(0^{+})=p^{\phantom'}_5(1)=\frac{30}{\pi^{4}}\int_0^\infty [I_{0}(t)]^{2}[K_0(t)]^4 t\D t.\label{eq:p6'_IKM}
\end{align}The corresponding Feynman diagram\begin{align}
\;\;\;\;\;
\dia{\put(-50,0){\line(-1,-1){50}}
\put(-50,0){\line(-1,1){50}}
\put(50,0){\line(1,1){50}}
\put(50,0){\line(1,-1){50}}
\put(0,0){\circle{100}}
\qbezier(-50, 0)(0, 35)(50, 0)
\qbezier(-50, 0)(0, -35)(50, 0)
\put(50,0){\vtx}
\put(-50,0){\vtx}
}{}\;\;=2^{3}\int_0^\infty [I_{0}(t)]^2[K_0(t)]^4 t\D t
\end{align} is equal to $4\pi^2 L(f_{4,6},1)=12L(f_{4,6},3)$, as conjectured by Broadhurst \cite[(110)]{Broadhurst2016} and verified in our recent work \cite[Theorem 4.2.3]{Zhou2017WEF}.
Thus, \eqref{eq:p6'_L} is true. [One can also represent $ p'_6(0^{+})$ using generalized hypergeometric series, based on a recently verified conjecture  \cite[(1.12)]{Zhou2018LaportaSunrise} of  Laporta \cite[(27)]{Laporta:2017okg} and Broadhurst (see \cite[\S2.2]{Broadhurst2017Paris}, \cite[\S2.2]{Broadhurst2017CIRM}, \cite[\S2.1]{Broadhurst2017Higgs}, \cite[\S3.1]{Broadhurst2017ESIa}, \cite[\S3.1]{Broadhurst2017DESY}).]

In \cite[Lemma 5.1.2]{Zhou2017WEF}, we have used Wick rotations to show that (see also Theorem \ref{thm:p7_Taylor} below) \begin{align}
\frac{p'_8(0^{+})}{35}=\frac{4}{\pi^6}\int_{0}^\infty[I_0(t)]^2[K_0(t)]^6t\D t-\frac{2}{\pi^4}\int_{0}^\infty[I_0(t)]^4[K_0(t)]^4t\D t.\label{eq:p8'_Wick}
\end{align}Meanwhile, the following conjectures of Broadhurst \cite[(142), (143),  (145)]{Broadhurst2016}\begin{align}
\frac{L(f_{6,6},5)}{L(f_{6,6},3)}={}&\frac{2\pi^2}{21},\\\int_{0}^\infty[I_0(t)]^4[K_0(t)]^4t\D t={}&L(f_{6,6},3),\\\int_{0}^\infty[I_0(t)]^2[K_0(t)]^6t\D t={}&\frac{27}{4}L(f_{6,6},5)
\end{align}have been confirmed in \cite[Theorems 5.1.1, 5.2.1,  5.2.2]{Zhou2017WEF}. With a reflection formula for $ L(f_{3,15},s)$ \cite[(138)]{Broadhurst2016} that relates $ L(f_{3,15},1)$ to $ L(f_{3,15},5)$, we conclude the proof of \eqref{eq:p8'_L}.\end{proof}

Here, the cusp forms $ f_{3,15}$, $ f_{4,6}$ and $ f_{6,6}$ occurring in the modular $L$-functions are not arbitrary: they arise from solutions to the corresponding Kloosterman problems (``Bessel moments over finite fields''), which have been investigated systematically by Broadhurst \cite[\S\S2--6]{Broadhurst2016}. Computations over finite fields determine local factors in the Hasse--Weil zeta functions, which piece together into the modular $L$-functions, namely, $ L(f_{3,15},s)$ for the 5-Bessel problem, $ L(f_{4,6},s)$ for the 6-Bessel problem, and $ L(f_{6,6},s)$ for the 8-Bessel problem. Numerical studies of these Kloosterman moments had enabled Broadhurst to discover many closed-form evaluations of individual Feynman diagrams \cite[\S7]{Broadhurst2016}, before their formal proofs were found \cite[\S\S3--5]{Zhou2017WEF}.

The Hasse--Weil zeta functions in Broadhurst's construction happen to result in modular $L$-functions when there are 5, 6, or 8 Bessel factors in the integrand.
For generic $n$, the Hasse--Weil $L$-function in Broadhurst's theory may not be modular, but still appears to be (as supported by strong numerical evidence \cite{Broadhurst2016,Broadhurst2017Paris,Broadhurst2017CIRM,Broadhurst2017Higgs}) good mathematical models for Feynman diagrams in 2-dimensional quantum field theory.

At present, we are unable to find further applications of Broadhurst's predictions to $ p_6(x)$ and $p_8(x) $, beyond their leading order asymptotic behavior. The major difficulty seems to reside in certain obstructions to implementing contour deformations.

Concretely speaking, for generic $x$, the  6-step density function $p_6(x)$ does not appear to be related to recognizable objects in 2-dimensional quantum field theory, at least not in any fashion that resembles $ p_3(x)$, $p_4(x)$ or $p_5(x)$. As $ p'_6(0^{+})$ is finite, the function $ p_6(x)$ differs qualitatively from $ p_4(x)=-\frac{3x}{2\pi^{2}}\log x+O(x),x\to0^+$ \cite[(4.4)]{BSWZ2012}.
Even though \eqref{eq:p6'_IKM} holds, numerical computations reveal that one cannot equate $ p_6(x)$ with \begin{align}
\frac{30}{\pi^{4}}\int_0^\infty I_{0}(xt)[I_{0}(t)]^{2}[K_0(t)]^4x t\D t
\end{align}for $ x>0$, contrary to the situations in  $ p_3(x)$ and $p_5(x)$.
While Wick rotation still
brings us [cf.~\eqref{eq:IIKKKK_Wick'}]\begin{align}
\left(\frac2\pi
\right)^4\int_0^\infty I_0(xt)[I_{0}(t)]^{2}[K_0(t)]^4 t\D t=-\int_{0}^{\infty}J_0(xt)J^{2}(J^{4}-6J^{2}Y^{2}+Y^{4}) t\D t
\end{align}for $ x\in[0,2]$ and $ J=J_0(t),Y=Y_0(t)$, we can no longer cancel out all the $Y$ factors in the integrand, as done in the proof of Lemma \ref{lm:p5_alt_int}. Unlike  \eqref{eq:Jordan_2}, we cannot close the contour upwards in the integral\begin{align}
\int_{i0^+-\infty}^{i0^++\infty}J_0(xz)[J_{0}(z)]^{3}[H_0^{(1)}(z)]^3z\D z
\end{align}when $x>0$, due to lack of exponential decay in the integrand as $ |z|\to\infty$, hence inapplicability of Jordan's lemma. One encounters a similar hurdle for $ p_8(x)$.

The story for the 7-step density function $ p_7(x)$  is quite different. On the analytic side, we have good news, as $ p_7(x)$ admits a convergent Maclaurin expansion, whose Taylor coefficients are all expressible as Bessel moments. On the arithmetic side, we have bad news, as none of these Bessel moments associated with $ p_7(x)$ are currently known to be related to special $L$-values.

\begin{theorem}[Taylor expansion for $ p_7(x)$]\label{thm:p7_Taylor}For $ x\in[0,1]$, we have \begin{align}
\frac{p_{7}(x)}{35}=\frac{4}{\pi^6}\int_{0}^\infty I_{0} (xt)I_0(t)[K_0(t)]^6t\D t-\frac{2}{\pi^4}\int_{0}^\infty I_{0} (xt)[I_0(t)]^3[K_0(t)]^4t\D  t.\label{eq:p7_IKM}
\end{align}Setting \begin{align}
r_{7,k}=\frac{140}{4^{k}(k!)^{2}\pi^6}\int_{0}^\infty \left\{I_0(t)[K_0(t)]^6-\frac{\pi^2}{2}[I_0(t)]^3[K_0(t)]^4\right\}t^{2k+1}\D  t
\end{align}for $ k\in\mathbb Z_{\geq0}$, we have a Maclaurin series\begin{align}
\sum^\infty_{k=0}r_{7,k}x^{2k+1}
\end{align}that converges uniformly to $ p_7(x)$ for $x\in[0,1]$. \end{theorem}\begin{proof}For $ x\in[0,1]$, direct applications of Wick rotations leave us\begin{align}
\left(\frac{2}{\pi}\right)^6\int_{0}^\infty I_{0} (xt)I_0(t)[K_0(t)]^6t\D t={}&\int_{0}^\infty J_{0} (xt)J(J^6-15 J^4 Y^2+15 J^2 Y^4-Y^6)t\D t,\label{eq:p7_Wick1}\\\left(\frac{2}{\pi}\right)^4\int_{0}^\infty I_{0} (xt)[I_0(t)]^{3}[K_0(t)]^4t\D t={}&-\int_{0}^\infty J_{0} (xt)J^{3}(J^4-6 J^2 Y^2+Y^4)t\D t.\label{eq:p7_Wick2}
\end{align}In parallel to \eqref{eq:Y_cancel}, the following computations{\allowdisplaybreaks\begin{align}\begin{split}&
J(J^6-15 J^4 Y^2+15 J^2 Y^4-Y^6)\\{}&-\frac{(J+i Y)^7-(-J+i Y)^7}{14} -J^2 [(J+i Y)^5-(-J+i Y)^5]\\={}&-\frac{8}{7}  J^5 (J^2-7 Y^2),\end{split}\\\begin{split}&J^{3}(J^4-6 J^2 Y^2+Y^4)-\frac{J^{2}}{10}[(J+i Y)^5-(-J+i Y)^5]\\={}&\frac{4}{5} J^5 (J^2-5 Y^2)\end{split}
\end{align}}allow us to reduce \eqref{eq:p7_Wick1} and \eqref{eq:p7_Wick2} into\begin{align}
\left(\frac{2}{\pi}\right)^6\int_{0}^\infty I_{0} (xt)I_0(t)[K_0(t)]^6t\D t={}&-\frac{8}{7}  \int_{0}^\infty J_{0} (xt)J^5 (J^2-7 Y^2)t\D t,\label{eq:p7_Wick1'}\\\left(\frac{2}{\pi}\right)^4\int_{0}^\infty I_{0} (xt)[I_0(t)]^{3}[K_0(t)]^4t\D t={}&-\frac{4}{5}\int_{0}^\infty J_{0} (xt)J^5 (J^2-5 Y^2)t\D t,\label{eq:p7_Wick2'}
\end{align}by virtue of Jordan's lemma. Eliminating the $Y$ terms from \eqref{eq:p7_Wick1'} and \eqref{eq:p7_Wick2'}, we arrive at \eqref{eq:p7_IKM}, which also incorporates the integral representation for  $p_8'(0^{+})=p_7^{\phantom'}(1) $ in~\eqref{eq:p8'_Wick} as a special case.

The rest can be verified by routine generalizations of the arguments in \S\ref{sec:Maclaurin}. \end{proof}We wrap up this section with two comments on the last theorem.  First, one can verify numerically that $ r_{7,0}>0$ and $ r_{7,1}<0$, so the non-negativity of Taylor coefficients for $ p_3(x)$  and $ p_5(x)$ no longer persists in $ p_7(x)$. Second, among Feynman diagrams involving 7 Bessel factors, only $ \int_0^\infty [I_0(t)]^2[K_0(t)]^5t\D t$ is known (numerically) to be expressible \cite[(129)]{Broadhurst2016} via a special value of a Hasse--Weil $L$-function (associated with a Hecke eigenform  of weight 3 and level 525  \cite[\S5.2]{Broadhurst2016}), so  the arithmetic nature of \begin{align} r_{7,0}^{\phantom'}=p_7'(0^{+})=p_6^{\phantom'}(1)=\frac{140}{\pi^6}\int_{0}^\infty \left\{I_0(t)[K_0(t)]^6-\frac{\pi^2}{2}[I_0(t)]^3[K_0(t)]^4\right\}t\D  t\end{align} remains an open question.
\section{Maclaurin expansions for $ p_{2j+1}(x)$ where $j\in\mathbb Z_{>0}$\label{sec:pn_Mac}}
So far, we have seen that the Taylor series for $ p_3(x)$, $ p_5(x)$ and $p_7(x)$ on  $[0,1)$  can be derived from their associated Feynman diagrams, in \eqref{eq:p3_IKM}, \eqref{eq:p5_IKM} and \eqref{eq:p7_IKM}, respectively. The derivations for these alternative integral representations of Kluyver's probability densities can be streamlined by the following algebraic identities:\begin{align}\frac{2J^{3}}{3}={}&-\frac{c_{3}}{6}+\frac{Jc_{2}}{2},\\-\frac{8 J^5}{15} ={}&-\frac{c_5}{10}+\frac{Jc_4}{2}-\frac{2J^{2}c_3}{3},\\
\frac{16 J^7}{35}={}&-\frac{c_{7}}{14}+\frac{Jc_6}{2}-\frac{6J^2c_5}{5}+J^{3}c_4,
\end{align}where $ c_\ell=(J+iY)^\ell+(J-iY)^\ell$. By Wick rotation, we have\begin{align}
\int_0^\infty I_0(xt)[I_0(t)]^{2m+1}\left[\frac{2K_{0}(t)}{\pi}\right]^{2(j-m)}t\D t=\frac{(-1)^{j-m+1}}{2}\int_0^\infty J_0(xt)J^{2m+1}c_{2(j-m)}t\D t
\label{eq:Bj_odd}\end{align} for $ 0\leq x\leq 1$, when $ m\in\mathbb Z\cap[0,(j-1)/2]$ for $ j\in\mathbb Z_{>1}$. By closing the contour upwards, we have \begin{align}\begin{split}
0={}&\int_{i0^+-\infty}^{i0^++\infty}J_0(xz)[J_{0}(z)]^{2m'}[H_0^{(1)}(z)]^{2(j-m')+1}z\D z\\={}&\int_0^\infty J_0(xt)J^{2m'}c_{2(j-m')+1}t\D t,\end{split}\label{eq:Bj_even}
\end{align}for $ 0\leq x\leq 1$, when $ m'\in\mathbb Z\cap[0,j/2]$ for $ j\in\mathbb Z_{>1}$.  (When $ j=1$, the conditions for the two equations above
need to be modified into $ 0\leq x<1$.)

Generalizing further, we arrive at the following theorem.

\begin{theorem}[$ p_{2j+1}(x)$ as Feynman diagrams]For each $ j\in\mathbb Z_{>1}$, the function $p_{2j+1}(x),0\leq x\leq 1 $ is a unique $ \mathbb Q$-linear combination of \begin{align}
\int_0^\infty I_0(xt)[I_0(t)]^{2m+1}\left[\frac{K_{0}(t)}{\pi}\right]^{2(j-m)}xt\D t,\quad\text{where } m\in\mathbb Z\cap[0,(j-1)/2].\label{eq:pn_IKM}
\end{align} (When $ j=1$, the same is true for  $ 0\leq x<1$.)\end{theorem}\begin{proof}First, the functions  listed in   \eqref{eq:pn_IKM} are linearly independent over $ C^\infty(0,1)$, as can be verified by a Wro\'nskian computation (see \cite[\S4]{Zhou2017BMdet} or \cite[\S2.2 and \S4]{Zhou2018ExpoDESY}).

Second, we show that the function $ J^{2j+1}=[J_0(z)]^{2j+1}$ is always a unique $ \mathbb Q$-linear combination of the following set with  $ (j+1)$ members\begin{align}B_j:=
\{c_{2j+1},Jc_{2j},\dots,J^jc_{j+1}\}\equiv B_j^{\mathrm e}\cup B_j^{\mathrm o},
\end{align}where $ B_j^{\mathrm e}$ (resp.~$ B_j^{\mathrm o}$) denotes a subset with even (resp.~odd) powers for $J$ and odd (resp.~even) subscripts for $c$. Clearly, each member in the set $ B_j$ resides in a $ (j+1)$-dimensional  $ \mathbb Q$-vector space spanned by $ \{J^{2j+1},J^{2j-1}Y^2,\dots,JY^{2j}\}$.  We only need to verify that $ B_j$ indeed forms a basis of this vector space, \textit{i.e.}~all the members in  $ B_j$ are linearly independent. An easy way to see this is to compute polynomial degree (in the variable $Y$)\begin{align}\begin{split}\deg_{Y}(J^kc_{2j+1-k})={}&
\deg_{Y}\{J^k[(J+iY)^{2j+1-k}+(J-iY)^{2j+1-k}]\}\\={}&\begin{cases}2j-k, & k\text{ even}, \\
2j+1-k , & k\text{ odd}, \\
\end{cases}\end{split}
\end{align}which immediately reveals both  $ B_j^{\mathrm e}$ and $ B_j^{\mathrm o}$ as linearly independent sets on their own.    Moreover,  we have $\Span B_j^{\mathrm e}\cap\Span B_j^{\mathrm o}=\{0\}$, according to  \eqref{eq:Bj_odd}, \eqref{eq:Bj_even} and the linear independence of the functions  listed in   \eqref{eq:pn_IKM}. Therefore, the dimension of the linear space spanned by  $ B_j$ is  $\dim\Span B_j=\dim \Span B_j^{\mathrm e}+\dim\Span B_j^{\mathrm o}= j+1$, as claimed.


Last, but not the least, depending on the parity of $k\in\mathbb Z\cap[0,j]$, the expression   $\int _{0}^{\infty}J_0(xt)J^kc_{2j+1-k}t\D t$ is either an integer multiple of  a Feynman diagram listed in   \eqref{eq:pn_IKM}, or a vanishing integral. This completes the proof of the alternative integral representation for  $p_{2j+1}(x),0\leq x\leq 1 $  as a linear combination of Feynman diagrams, which generalizes  \eqref{eq:p3_IKM}, \eqref{eq:p5_IKM} and \eqref{eq:p7_IKM}.  \end{proof}

The theorem above has some interesting consequences.
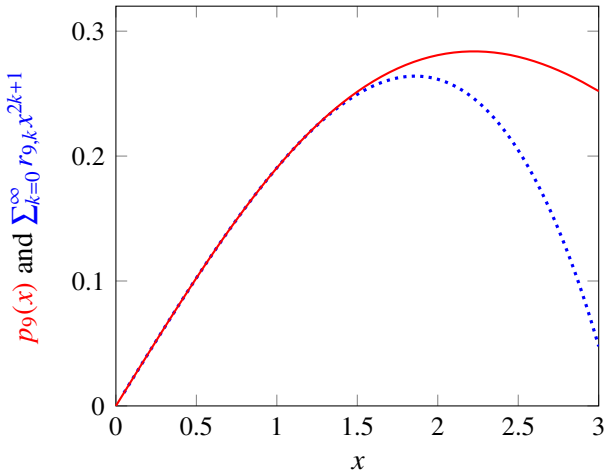
\begin{figure}[b]\begin{minipage}{0.5\textwidth}\begin{center}\begin{tikzpicture}
\pgfplotsset{width=8cm,tick label style={font=\footnotesize},label style={font=\small}}
\begin{axis}[xlabel={$x$},ylabel={$\textcolor{red}{ p_9(x)}$ and \textcolor{blue}{$ \sum^\infty_{k=0}r_{9,k}x^{2k+1} $}},ymin=0,ymax=.32,xmin=0,xmax=3,enlargelimits=false]
\addplot [
smooth, dotted,
draw=blue, very thick
] table {
x y0. 0.
 0.05 0.0104652
 0.1 0.0209165
 0.15 0.03134
 0.2 0.0417217
 0.25 0.0520478
 0.3 0.0623042
 0.35 0.0724769
 0.4 0.0825517
 0.45 0.0925146
 0.5 0.102351
 0.55 0.112047
 0.6 0.121588
 0.65 0.13096
 0.7 0.140147
 0.75 0.149135
 0.8 0.157909
 0.85 0.166454
 0.9 0.174755
 0.95 0.182796
 1. 0.190562
 1.05 0.198036
 1.1 0.205202
 1.15 0.212045
 1.2 0.218548
 1.25 0.224694
 1.3 0.230465
 1.35 0.235844
 1.4 0.240814
 1.45 0.245356
 1.5 0.249452
 1.55 0.253082
 1.6 0.256229
 1.65 0.258871
 1.7 0.260989
 1.75 0.262562
 1.8 0.263568
 1.85 0.263986
 1.9 0.263793
 1.95 0.262967
 2. 0.261483
 2.05 0.259318
 2.1 0.256444
 2.15 0.252838
 2.2 0.24847
 2.25 0.243314
 2.3 0.23734
 2.35 0.230518
 2.4 0.222816
 2.45 0.214202
 2.5 0.20464
 2.55 0.194095
 2.6 0.182528
 2.65 0.169899
 2.7 0.156166
 2.75 0.141283
 2.8 0.125202
 2.85 0.107869
 2.9 0.0892269
 2.95 0.0692111
 3. 0.0477459

 }
;
\addplot [
draw=red, thick
] table {
x y
0. 0.
 0.05 0.0104655
 0.1 0.020917
 0.15 0.0313407
 0.2 0.0417224
 0.25 0.0520482
 0.3 0.0623038
 0.35 0.0724755
 0.4 0.0825496
 0.45 0.0925126
 0.5 0.10235
 0.55 0.112048
 0.6 0.12159
 0.65 0.130978
 0.7 0.140162
 0.75 0.149137
 0.8 0.15791
 0.85 0.166454
 0.9 0.174754
 0.95 0.182794
 1. 0.190553
 1.05 0.198036
 1.1 0.205217
 1.15 0.212096
 1.2 0.218669
 1.25 0.224932
 1.3 0.230882
 1.35 0.236517
 1.4 0.241836
 1.45 0.246835
 1.5 0.251518
 1.55 0.255879
 1.6 0.259918
 1.65 0.263643
 1.7 0.267048
 1.75 0.270133
 1.8 0.272898
 1.85 0.275346
 1.9 0.277479
 1.95 0.279298
 2. 0.280806
 2.05 0.282005
 2.1 0.282899
 2.15 0.28349
 2.2 0.283782
 2.25 0.28378
 2.3 0.283487
 2.35 0.282909
 2.4 0.282049
 2.45 0.280958
 2.5 0.279502
 2.55 0.277823
 2.6 0.275884
 2.65 0.273692
 2.7 0.271256
 2.75 0.268583
 2.8 0.265683
 2.85 0.262566
 2.9 0.25924
 2.95 0.255587
 3. 0.251996
  };

\end{axis}
\end{tikzpicture}\end{center}\end{minipage}\begin{minipage}{0.5\textwidth}
\caption{ Kluyver's 9-step density function $ p_9(x)$ (\textit{red solid}  curve) and the  Maclaurin series $ \sum^\infty_{k=0}r_{9,k}x^{2k+1}=\frac{630}{\pi^{8}}\int_0^\infty I_0(xt)I_{0}(t)[K_0(t)]^8 xt\D t-\frac{840}{\pi^{6}}\int_0^\infty I_0(xt)[I_{0}(t)]^{3}[K_0(t)]^6 xt\D t$ (\textit{blue dotted}  curve) for $0\leq x\leq 3 $. Note that $ p_9(x)$ is not real analytic at $x=1$ \cite[Theorem 2.4]{BSWZ2012}, and it overlaps the Maclaurin series only for $ x\in[0,1]$, even though the pixelation in the graph may create an illusion that there is still agreement for  certain arguments $ x>1$.\label{fig:p9_Taylor}}\end{minipage}\end{figure}
\begin{theorem}[Taylor expansion for $ p_{2j+1}(x)$]\label{thm:pn_Mac}For each  $ j\in\mathbb Z_{>1}$, there exists a uniformly convergent Maclaurin series:\begin{align}
p_{2j+1}(x)=\sum^\infty_{k=0}r_{2j+1,k}x^{2k+1},\quad 0\leq x\leq 1.\label{eq:pn_Taylor}
\end{align}(For the special case where $j=1$, the expression $ p_3(1^-)$ diverges, and the power series $ \sum^\infty_{k=0}r_{3,k}x^{2k+1}$ converges uniformly for $ 0\leq x\leq \delta$ with $ \delta\in(0,1)$.)

The analytic continuation of the  corresponding ramble integral $ W_{2j+1}(s)=\int_0^\infty x^s p_{2j+1}(x)\D x$ has only simple poles.

Furthermore, for each $j\in\mathbb Z_{>0} $, the infinite series\begin{align}
\sum^\infty_{k=0}\frac{r_{2j+1,k}}{s+2k+2}\label{eq:abs_0}
\end{align}converges absolutely and uniformly, when $s$ is restricted to any compact subset of $\mathbb C\smallsetminus(2\mathbb Z_{<0}) $.  \end{theorem}\begin{proof}As before, the Taylor coefficient $r_{2j+1,k} $ is attributable to  Bessel moments associated with finitely many Feynman diagrams listed in   \eqref{eq:pn_IKM}. Concretely speaking,  for each  $ j\in\mathbb Z_{>1}$, we may decompose the Taylor coefficient $ r_{2j+1,k}$ into \begin{align}
r_{2j+1,k}^{\phantom{(m)}}=\sum_{m\in\mathbb  Z\cap[0,(j-1)/2]}r_{2j+1,k}^{(m)},
\end{align} where $r_{2j+1,k}^{(m)}  $ is a constant rational multiple of \begin{align}\frac{1}{4^k(k!)^2}
\int_0^\infty [I_0(t)]^{2m+1}\left[\frac{K_{0}(t)}{\pi}\right]^{2(j-m)}t^{2k+1}\D t.
\end{align} Here,  for each $ m\in\mathbb  Z\cap[0,(j-1)/2]$, the sequence $ \{r_{2j+1,k}^{(m)} |k\in\mathbb Z_{\geq0}\}$ carries a  fixed sign. Accordingly,  we can use Dini's theorem  to show that the sequence of continuous functions   $ \sum^n_{k=0}r_{2j+1,k}^{(m)}x^{2k+1}$ converges uniformly to  $ \sum^\infty_{k=0}r_{2j+1,k}^{(m)}x^{2k+1}$ on $ [0,1]$. Thus, the  Maclaurin series in \eqref{eq:pn_Taylor} is a finite sum of uniformly convergent continuous functions, converging  to $ p_n(x),0\leq x\leq 1$,  as stated.
(This does not preclude, however, the possibility that the Maclaurin series has a radius of convergence  greater than $1$, as shown in Figs.~\ref{fig:p5_Taylor} and \ref{fig:p9_Taylor}.)

We note that in the decomposition  $ W_{2j+1}(s)=\int_0^1 x^s p_{2j+1}(x)\D x+\int_1^{2j+1} x^s p_{2j+1}(x)\D x$, the second integral contributes no singularities in the complex $ s$-plane, so the pole structure for the analytic continuation of the ramble integral $ W_{2j+1}(s)$ is completely determined by $ p_{2j+1}(x),0\leq x\leq 1$ (or the asymptotic behavior in any non-void open neighborhood of the origin).
After Mellin inversion, a pole $ 1/(s-s_0)^{k+1}$ at a negative integer $s_0\in\mathbb Z_{<0}$ translates into a term $ \frac{(-1)^k}{k!}x^{-s_0}\log^k x$ in the (generalized) power series \cite[Appendix B.7]{FlajoletSedgewick2009}. Therefore,  the analytic continuation of $ W_{2j+1}(s)$ contains only simple poles. (See also \eqref{eq:W_part_frac} below.)

 For each $j\in\mathbb Z_{>0}$, by Levi's monotone convergence theorem, we have\begin{align}
\int_0^1\sum^\infty_{k=0}r_{2j+1,k}^{(m)}x^{2k+1}\D x=\sum^\infty_{k=0}\frac{r_{2j+1,k}^{(m)}}{2k+2}=\frac{r_{2j+1,0}^{(m)}}{\left|r_{2j+1,0}^{(m)}\right|}\sum^\infty_{k=0}\frac{\left|r_{2j+1,k}^{(m)}\right|}{2k+2},\label{eq:abs_conv}
\end{align}and the right-hand side of the equation above is a finite real number. For  $s$ residing in a  compact subset of   $\mathbb C\smallsetminus(2\mathbb Z_{<0})$,  the series\begin{align}
\sum^\infty_{k=0}\left\vert\frac{r_{2j+1,k}^{(m)}}{s+2k+2}\right\vert\label{eq:abs_sum_s}
\end{align}   converges uniformly, through comparison with the absolutely convergent series   \eqref{eq:abs_conv}, in a Weierstra{\ss} $M$-test. The expression in \eqref{eq:abs_0}, as a sum of finitely many  well-behaved series, also converges absolutely and compactly.   \end{proof}\begin{theorem}[A sum rule for ramble integrals]For $j\in\mathbb Z_{>0} $ and $ \nu\in\mathbb C\smallsetminus\mathbb Z$, we have \begin{align}
W_{2j+2}(\nu)=\sum_{m=0}^\infty\left[\frac{\Gamma\left(\frac{\nu}{2}+1\right)}{\Gamma(m+1)\Gamma\left(\frac{\nu}{2}-m+1\right)}\right]^2W_{2j+1}(\nu-2m ).\label{eq:W_sum}
\end{align}Moreover, the formula above extends to all $ \nu\in\mathbb C\smallsetminus(2\mathbb Z_{<0})$, by continuity. When $ \nu/2\in\mathbb Z_{\geq0}$, the series on the right-hand side truncates to a finite sum, due to poles of $ \Gamma\left(\frac{\nu}{2}-m+1\right)$. \end{theorem}\begin{proof}We shall only prove the sum rule for  $ s=-\nu\in(1/2,1)$, as the rest follows from analytic continuation.

As pointed out by Borwein--Straub--Wan \cite[(4.2)]{BSW2013}, it would suffice to show that, in the following Mellin inversion formula\begin{align}
W_{2j+2}(-s)=\frac{\Gamma\left( 1-\frac{s}{2} \right)}{\Gamma\left( \frac{s}{2}\right)}\frac{1}{2\pi i}\int_{\frac{1}{2}-i\infty}^{\frac{1}{2}+i\infty}\frac{\Gamma\left( \frac{z}{2} \right)\Gamma\left( \frac{s-z}{2} \right)}{2\Gamma\left( 1-\frac{z}{2} \right)\Gamma\left( 1-\frac{s-z}{2} \right)}W_{2j+1}(-z)\D z,
\end{align}the contour can be closed to the right, with residue contributions only coming from the poles of $ \Gamma\left( \frac{s-z}{2} \right)$.

To justify their claims of contour deformation, we need two observations.

First, judging from Theorem \ref{thm:pn_Mac},  the only possible singularities of $W_{2j+1}(-z) $ are simple poles at $ z\in2\mathbb Z_{>0}$, the effects of which will be canceled by the corresponding zeros of $ 1/\Gamma\left( 1-\frac{z}{2} \right)$, so there are no residue contributions from the singularities of $W_{2j+1}(-z) $  to the right-hand side of  \eqref{eq:W_sum}.

Second, for  semi-circular contours  $ \gamma_r:=\left\{z\in\mathbb C\left|\R z>\frac{1}{2},\left|z-\frac{1}{2}\right|=r\right.\right\}$, we will show that \begin{align}
\varliminf_{r\to\infty}\left|\int_{\gamma_r}\frac{\Gamma\left( \frac{z}{2} \right)\Gamma\left( \frac{s-z}{2} \right)}{\Gamma\left( 1-\frac{z}{2} \right)\Gamma\left( 1-\frac{s-z}{2} \right)}W_{2j-1}(-z)\D z\right|=0.\label{eq:liminf0}
\end{align}With the uniformly convergent Taylor expansion in \eqref{eq:pn_Taylor}, we have \begin{align}\begin{split}
W_{2j+1}(\alpha)={}&\int_0^1 x^\alpha p_{2j+1}(x)\D x+\int_1^{2j+1} x^\alpha p_{2j+1}(x)\D x\\={}&\sum^\infty_{k=0}\frac{r_{2j+1,k}}{\alpha+2k+2}+\int_1^{2j+1} x^\alpha p_{2j+1}(x)\D x\end{split}
\end{align}for $  \alpha>0$. (Here, Levi's monotone convergence theorem ensures that  termwise integration is permissible.) After analytic continuation, we have a partial fraction expansion:\begin{align}
W_{2j+1}(-z)-\int_1^{2j+1} x^{-z} p_{2j+1}(x)\D x=\sum^\infty_{k=0}\frac{r_{2j+1,k}}{2k+2-z},\label{eq:W_part_frac}
\end{align} valid for all complex-valued $z$ except the positive even integers. In view of the absolute and compact convergence properties in Theorem  \ref{thm:pn_Mac}, the infinite series on the right-hand side of the equation above defines a meromorphic function for $ z\in\mathbb C\smallsetminus(2\mathbb Z_{>0})$, and is bounded by a constant when $z$ runs to infinity in a sequence of concentric circles  $\left|z-\frac{1}{2}\right|=2k+s,k\in\mathbb Z_{\geq0} $.
These asymptotic bounds allow us to prove \eqref{eq:W_sum} for each fixed $ s\in(1/2,1)$, as we choose a sequence of  semi-circles $ \gamma_r=\gamma_{2k+s}$, on which \begin{align}\begin{split}
\frac{\Gamma\left( \frac{z}{2} \right)\Gamma\left( \frac{s-z}{2} \right)}{\Gamma\left( 1-\frac{z}{2} \right)\Gamma\left( 1-\frac{s-z}{2} \right)}={}&\frac{\sin\frac{\pi z}{2}}{\sin\frac{\pi(s- z)}{2}}\left[ \frac{\Gamma\left( \frac{z}{2} \right)}{\Gamma\left( 1-\frac{s-z}{2} \right)} \right]^2\\={}&\frac{1}{\sin\frac{\pi s}{2}\cot\frac{\pi z}{2}-\cos\frac{\pi s}{2}}\left( \frac{2}{z} \right)^{2-s}\left[ 1+O\left( \frac{1}{z} \right) \right]=O\left( \frac{1}{z^{2-s}} \right)\end{split}
\end{align} is sufficient to close the contour in \eqref{eq:liminf0}  rightwards.\end{proof}Finally, we remark that  \eqref{eq:W_sum} for the situation $ j=1,\nu\in\mathbb Z$ has been singled out in \cite[Theorem 7.7]{BSWZ2012}.
\subsection*{Acknowledgements}I thank David Broadhurst and Wadim Zudilin for their critical reading of the initial draft for this manuscript.


\end{document}